%% file: PSL2C_-_final.tex
\documentclass[11pt]{article}      
\usepackage{graphicx,color}
\usepackage{latexsym}
\usepackage{mathtools}
\usepackage{amssymb}
\usepackage{mathbbol}
\usepackage{amsopn}
\usepackage{amsfonts}  
\usepackage{amsthm}
\usepackage{mathrsfs}
\usepackage{amssymb}
\usepackage{titlesec}
\usepackage{amscd}
\usepackage{ushort}
\usepackage{stmaryrd}
\usepackage{mathbbol}
\usepackage{authblk}
\usepackage{mathdots}

\usepackage{tikz}
\usetikzlibrary{matrix,arrows}
\tikzset{node distance=2cm, auto}

\makeatletter
\newsavebox{\@brx}
\newcommand{\llangle}[1][]{\savebox{\@brx}{\(\m@th{#1\langle}\)}%
  \mathopen{\copy\@brx\kern-0.5\wd\@brx\usebox{\@brx}}}
\newcommand{\rrangle}[1][]{\savebox{\@brx}{\(\m@th{#1\rangle}\)}%
  \mathclose{\copy\@brx\kern-0.5\wd\@brx\usebox{\@brx}}}
\makeatother

\ushortCreate()[-.35\ht0]{ushort}

\def \vphi {\varphi}
\def \eps {\varepsilon}

\def \vW {\mathcal{W}}

\def \an {\mathrm{an}}

\def \cC {\mathbb{R}}
\def \cPC {\cP^1(\cC)}
\def \cN {\mathbb{N}}
\def \cC {\mathbb{C}}
\def \cQ {\mathbb{Q}}
\def \cR {\mathbb{R}}

\def \cZ {\mathbb{Z}}

\def \cP {\mathbb{P}}

\def \vF {\mathcal{F}}

\def \Const {\mathrm{Const}}

\def \Der {{\mathrm{Der}}}

\def \< {{\langle}}
\def \> {{\rangle}}

\def \Ker {{\mathrm{Ker}}}
\def \Im {{\mathrm{Im}}}
\def \Re {{\mathrm{Re}}}

\def \Paths {{{\Lambda}}}

\newtheorem{Theorem}{Theorem}[section]

\newtheorem*{MainTheorem}{Main Theorem}
\newtheorem{Definition}[Theorem]{Definition}
\newtheorem{Proposition}[Theorem]{Proposition}
\newtheorem{Lemma}[Theorem]{Lemma}
\newtheorem{Corollary}[Theorem]{Corollary}

\newtheorem{Rem}[Theorem]{Remark}
\newtheorem{Rems}[Theorem]{Remarks}
\newtheorem{Exam}[Theorem]{Example}
\newtheorem{Exams}[Theorem]{Examples}
\newenvironment{Remark} {\begin{Rem} \sc \rm }{\end{Rem}}
\newenvironment{Remarks} {\begin{Rems} \sc \rm }{\end{Rems}}
\newenvironment{Example} {\begin{Exam} \rm }{\end{Exam}}

\def \Merom {\mathcal{M}}

\def \Diff {\mathrm{Diff}}

\def \NF {\mathrm{NF}}
\def \NFPT {\mathrm{NFPT}}
\def \BNF {\mathrm{BNF}}

\def \PNF {\mathrm{PNF}}
\def \fin {\mathrm{finite}}
\def \arg {\mathrm{arg}}
\def \trans {\mathrm{T}}

\let\oldmarginpar\marginpar
\renewcommand\marginpar[1]{\-\oldmarginpar[\raggedleft\footnotesize #1]%
{\raggedright\footnotesize #1}}

\def \maxn {\mathfrak{n}}

\def\id {\mathbb{1}}
\def\rightstar {\buildrel * \over \rightarrow}
\def\leftstar {\buildrel * \over \leftarrow}
\def\leftrightstar {\buildrel * \over \leftrightarrow}

\def\het {\mathrm{height}}

\DeclareMathAlphabet{\mathpzc}{OT1}{pzc}{m}{it}

\def \tl {\boldsymbol{l}}
\def \te {\boldsymbol{e}}
\def \ts {\boldsymbol{s}}
\def \tt {\boldsymbol{t}}
\def \ti {\boldsymbol{w}}

\def \tp {\boldsymbol{p}}

\def \ttheta {\boldsymbol{\theta}}

\def \Free{\mathpzc{F}}

\def \Gr{{\mathpzc{G}}}
\def \PGr{{\Gamma\mathpzc{G}}}
\def \NGr{{\mathpzc{N}}}
\def \HGr{{\mathpzc{H}}}
\def \LGr{{\mathpzc{L}}}

\def \Rel {\mathrm{Rel}}
\def \source {\mathrm{s}}
\def \target {\mathrm{t}}

\def \lex {\mathrm{lex}}

\def \transdeg {\mathrm{tr.deg.}}

\def \Aff{\mathrm{Aff}}
\def \Pow{\mathrm{Pow}}

\def \fexp{\mathbb{exp}}
\def \texp{\Phi}
\def \Coll {\mathcal{C}}
\def \Germ {\mathrm{Germ}}
\def \PSL{{\mathrm{PSL}(2,\cC)}}

\def \Exp{\mathrm{Exp}}

\def \gr{G}
\def \Homeo {\mathrm{Homeo}}

\def \stable{\mathtt{k}}
\def \Obj{\mathrm{Obj}}

\def \gtranss {\mathscr{P}}
\def \transs {\mathbb{T}}

\def \monoid {\mathcal{M}}

\def \aug {\mathrm{aug}}

\def \tame {\mathrm{tame}}

\makeatletter
\def\@tvsp{\mathchoice{{}\mkern-5.5mu}{{}\mkern-5.5mu}{{}\mkern-3.5mu}{}}
\def\ltriple{\left[\@tvsp\left[\@tvsp\left[}
\def\rtriple{\right]\@tvsp\right]\@tvsp\right]}
\makeatother

\begin{document}

\title{$\PSL$, the exponential and some new free groups}

\author[D. Panazzolo]{Daniel Panazzolo \footnote{This work has been partially supported by the ANR project STAAVF and by the CAPES/COFECUB project MA731-12.}}

\newcommand{\Addresses}{{
  \bigskip
  \footnotesize

  \textsc{Laboratoire de Math\'{e}matiques, Informatique et Applications, Universit\' e de Haute-Alsace, France}\par\nopagebreak
  \textit{E-mail address}: \texttt{daniel.panazzolo@uha.fr}

  \medskip

  \textsc{Universit\' {e} de Strasbourg, France}

}}

\date{}
\maketitle
\abstract{We prove a normal form result for the groupoid of germs generated by $\PSL$ and the exponential map.  
We discuss three consequences of this result: (1) a generalization of a result of Cohen about the group of translations and powers, which gives a positive answer to a problem posed by Higman; (2) a proof that the subgroup of 
$\Homeo(\cR,+\infty)$ generated by the positive affine maps and the exponential map is isomorphic to a HNN-extension; (3) a {\em finitary }version of the immiscibility conjecture of Ecalle-Martinet-Moussu-Ramis}
\tableofcontents
\section{Introduction}
\subsection{Normal forms}\label{subsect-intronormalforms}
We recall some basic concepts and terminology from the theory of groupoids (see e.g.\ \cite{MR2273730}).  A groupoid is a category $\Gr$ whose objects $\Obj(\Gr)$ form a set and in which every morphism is an isomorphism. For each $x,y \in \Obj(\Gr)$, we denote by $\Gr(x,y)$ the set of morphisms in $\Gr$ from $x$ to $y$. We denote also by $\Gr$ the disjoint union of $\Gr(x,y)$, for all $x,y \in \Obj(\Gr)$. The composition of morphism is written multiplicatively: if $f \in \Gr(x,y)$ and $g \in \Gr(y,z)$, then these morphisms can be composed and its composition is the morphism $g f \in \Gr(x,z)$.  From now on, when we write the expression $g f$ for two morphisms $f,g$, we are tacitly assuming that these morphism can be composed.  The symbol $\id$ will generally denote the identity morphism and $f^{-1} \in \Gr(y,z)$ will denote the inverse of the morphism $f \in \Gr(x,y)$. 
We will say that morphism $f$ has
{\em source} $x = \source(f)$ and  {\em target}  $y = \target(f)$ if  $f \in \Gr(x,y)$. The group $\Gr(x,x)$ will 
be called {\em vertex group } at $x$ and will be denoted simply by $\Gr(x)$. 

A finite sequence of morphisms $[f_1, f_2, \ldots, f_n]$ in $\Gr$ is called a {\em path} if $\source(f_i) = \target(f_{i+1})$.   
Given such a path, we will say that 
$f = f_1 \cdots f_n \in \Gr$ is the morphism {\em defined} by the path.  
The operation of concatenation in the set of paths is defined in the obvious way, 
taking into account the source/target compatibility. 

A path $[f_1, f_2, \ldots, f_n]$ is called {\em reduced }if:
\begin{itemize}
\item no two consecutive morphisms $f_i$, $f_{i+1}$ are mutually inverse.
\item if some $f_i$ is the identity morphism then $n = 1$ and $f = [\id]$. 
\end{itemize} 
We can give a groupoid structure to the set of reduced paths.   The operation of composition of two paths is defined as 
follows: first concatenate the paths and then successively eliminate all consecutive terms which are mutually inverses.  
The resulting groupoid is called the {\em free groupoid on the graph of $\Gr$} (\cite{MR2273730}, section 8.2).

Given a differentiable manifold $M$, let $\Gr(M)$ denote the {\em Haefliger groupoid over $M$} (see e.g.~\cite{MR2012261}, section 5.5).   We recall that, by definition, the set of objects $\Obj(\Gr(M))$ is the set of points of $M$ and, for each $p,q \in M$, $\Gr(M)(p,q)$ is  
the set of all germs of diffeomorphisms $(M,p) \rightarrow (M,q)$.  In order to keep the traditional naming, we will refer to the morphisms of $\Gr(M)$ simply as {\em germs}.

Given a map $f: U \rightarrow V$, where $U,V \subset M$ are open sets and $f$ is a local diffeomorphism (i.e.~locally invertible), we denote by
$\Germ(f) \subset \Gr(M)$ the smallest wide subgroupoid containing all the germs $f_{,p}$ of $f$ at all points $p$ of its domain.  We recall that a subgroupoid $\Gr_1$ of a groupoid $\Gr_2$ is called {\em wide} if $\Obj(\Gr_1) = \Obj(\Gr_2)$.

More generally, given an arbitrary collection $\Coll$ of local diffeomorphisms as above, we denote by $\Germ(\Coll) \subset \Gr(M)$ the smallest subgroupoid containing $\Germ(f)$, for all $f \in \Coll$.

From now on, we shall assume that $M = \cPC$ and that all maps are holomorphic.  As a basic object, we will frequently consider the groupoid $\Gr_{\Exp} = \Germ(\exp)$ associated to the usual exponential map.  This is the groupoid whose germs at each point are given by
finite compositions  $f = f_1 \cdots f_n$ of the following germs
$$
\{\id_{,p}\} \cup \{\exp_{,p} : \text{if } p \in \cC\} \cup \bigcup_{k \in \cZ} \{\ln_{k,p} : \text{ if }p \in \cC^*\}
$$
where $\id_{,p}$ is the the identity germ at $p$ and $\ln_{k,p}$ is the germ at $p$ of the $k^{th}$-branch of the logarithm, i.e. the map
\begin{align*}
\ln_k : \cC^* &\longrightarrow J_k = \{x + i y : y \in ](2k-1)\pi i, (2k+1)\pi i] \}\\
z &\longmapsto \ln(|z|) + i\, \arg_k(z)
\end{align*}
where $\arg_k: \cC \rightarrow ](2k-1)\pi i, (2k+1)\pi i]$ is the $k^{th}$-branch of the argument function.  In general, we have the relation 
$$\exp_{,q} \log_{k,p} = \id_{,q}$$ 
for all $p \in \cC$ and $q = \log_{k}(p)$.  On the other hand, 
$$\log_{k,q} \exp_{,p} : z \mapsto z  + 2\pi i (k-s)$$ 
for all $p \in J_s$ and $q = \exp(p)$.  In particular, notice that germ corresponding to the translation by $2\pi i$ lies in 
$\Gr_{\Exp}$.  

\begin{figure}[htbp]
\begin{center}
{ 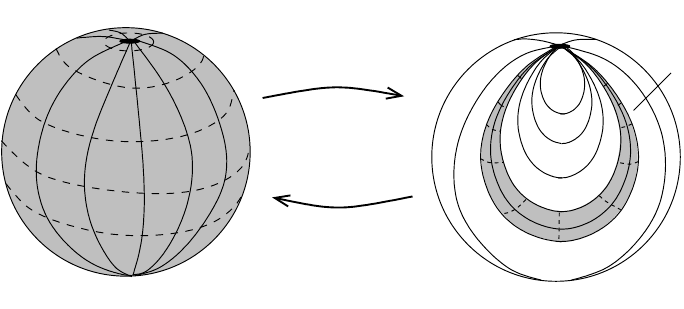}
\label{exponentialfol}

\end{center}
\end{figure}

In what follows, we are going to simplify the notation and omit the subscripts ${}_{,p}$ when referring to the 
germ of a local diffeomorphism at a point $p$ of its domain.  Thus, the same symbol, say $\exp$, will denote both the exponential 
map and the germ at each point of its domain.  In the situation where we want to emphasize that we are considering its germ 
at a specific point $p$, we will simply write that $\source(\exp) = p$. 
We also introduce the following symbols for the (germ of) exponential map and the zeroth branch of the logarithm:
$$
\te : z \mapsto \exp(z), \quad \tl : z \mapsto \ln_0(z).
$$
Another important object for us is the groupoid 
$$\Gr_{\PSL} = \Germ(\PSL).$$
We recall that the group $\PSL$ is generated by the subgroups 
\begin{align*}
W &= \{\id: z \mapsto z, \ \ti: z \mapsto 1/z\}, \quad
T = \{\tt_a : z \mapsto z + a ,\;  a \in \cC \} \\ 
S &= \{\ts_\alpha: z \mapsto \alpha z ,\; \alpha \in \cC^*\}
\end{align*}
which are, respectively,  the {\em involution}, the {\em translations} and the {\em scalings}.  We denote by $\Aff \subset \PSL$ the subgroup of affine maps, i.e.~$\Aff = T \rtimes S$.  

Following the above notational convention, the same symbols $\tt_a,\ts_\alpha$ and $\ti$ will be used to denote the corresponding germs in $\Gr_{\PSL}$. 

Our main result is a normal form for elements in the groupoid  
$$
\Gr_{\PSL,\Exp} = \Germ(\PSL \cup \{\exp\})
$$
In order to state this result, consider the subgroups $H_0,H_1 \subset \PSL$ given by
$$H_0 = T \rtimes \{\ts_{-1}\}, \quad\text{ and }\quad H_1 = S \rtimes \{\ti\}.$$
For the next definition, we recall that, given a group $G$ and a subgroup $H \subset G$, a {\em right transversal for $H$} is a subset $\trans \subset G$ of representatives for the right cosets $\{Hg : g \in G\}$ which contains the identity of $G$.
\begin{Definition}\label{def-nf}
Let $\trans_0,\trans_1 \subset \PSL$ be right transversals for $H_0,H_1$, respectively.  A {\em $(\trans_0,\trans_1)$-normal form} in $\Gr_{\PSL,\Exp}$ is a path
$$
g = [g_0, h_1,g_1, \ldots, h_n,g_n], \quad n \ge 0
$$
such that  the following conditions hold:
\begin{itemize}
\item[(i)] The germ $g_0$ lies in $\Gr_\PSL$.
\item[(ii)] For each $1 \le i \le n$, $h_i \in \{\te,\tl\}$.
\item[(iii)]  If $h_i = \te$ then $g_i \in \trans_0$.
\item[(iv)]  If $h_i = \tl$ then $g_i \in \trans_1$.
\item[(v)]  There are no subpaths of the form $[\te,\id,\tl]$ or $[\tl,\id,\te]$.
\end{itemize}
We denote by $\NF_{\trans_0,\trans_1}$ the set of all $(\trans_0,\trans_1)$-normal forms.  The path $[\id]$ will be called the {\em identity normal form}.
\end{Definition}
There is an obvious mapping 
$$\vphi: \NF_{\trans_0,\trans_1} \rightarrow \Gr_{\PSL,\Exp}$$ 
which associates to each normal form $g = [g_0,h_1,\ldots,g_n]$ the 
germ $\vphi(g) = g_0 h_1 \cdots g_n$. The main goal of this paper will be to study the surjectivity and injectivity properties of this mapping.
\begin{Remark}\label{remark-defNF}
As we shall see in Lemma~\ref{lemma-nfPSL}, a possible choice of transversals $\trans_0,\trans_1$ for $H_0,H_1$, respectively, is as follows:
\begin{align*}
\trans_0 &= \{\ts_\rho : \rho \in \Omega\} \cup \{\ts_{\rho} \ti \tt_b : \rho \in \Omega, b \in \cC\} \\
\trans_1 &= \{\tt_b : b \in \cC\} \cup \{\tt_{a} \ti \tt_b : a \in \Omega, b \in \cC^* \setminus \{-1/a\}\} \cup \{\tt_c \ti : c \in \cC^*\}
\end{align*}
where $\Omega  = \{\alpha : \Re(\alpha) > 0\} \cup \{\alpha : \Re(\alpha) = 0, \Im(\alpha) > 0\}$ is the region shown in 
figure~\ref{regionOmega}.\footnote{In fact, we could define similar transversals by choosing any region in $\cC^*$ which is a fundamental domain for the $\cZ_2$-action 
$z \mapsto -z$ and contains $1$.}

\begin{figure}[!h]
\centering
\begin{tikzpicture}[every node/.append style={font=\scriptsize}]

    \draw [->,black,line width=0.8pt] (-1.5,0) -- (1.5,0) ;
    \draw [->,black,line width=0.8pt] (0,0) -- (0,1.5) ;
    \draw [dashed,black,line width=1.2pt] (0,-1.5) -- (0,0) ;
    \draw (0.5,0) node {$\bullet$};
   
     \node[below] at (0.5,0) {$1$};
     \node[left] at (1,0.8) {$\Omega$};
     \node[above] at (0,1.5) {$\Im(\alpha)$};
     \node[right] at (1.5,0) {$\Re(\alpha)$};
     
     \draw[fill=gray,opacity=0.3]  (0,-1.5) -- (0,1.5) -- (1.5,1.5) -- (1.5,-1.5) -- cycle;
     
     \end{tikzpicture}
\caption{The region $\Omega$}
\label{regionOmega}
\end{figure}
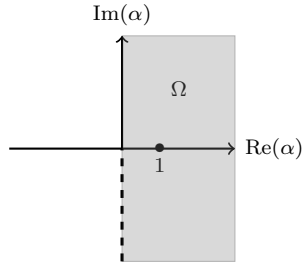

\end{Remark}
From now on, in order to simplify the exposition, we shall fix the choice of transversals $\trans_0,\trans_1$ as described in the remark \ref{remark-defNF},  and write $\NF_{T_0,T_1}$ simply as $\NF$.  Each result that we are going to discuss in the remaining of the paper can be appropriately translated to different choices of transversals.  

In order to state the Main Theorem, we need define certain special normal forms.  To simplify the notation, we shall frequently omit the square braces and write a path $[g_1,\ldots,g_n]$ simply as $g_1\cdots g_n$.

For each $\alpha \in \cC^*$, the {\em power map with exponent $\alpha$ }is the germ defined by
$$
\tp_\alpha = \te\, \ts_\alpha\, \tl,
$$
i.e.~the germ of power map $z \mapsto z^\alpha$ obtained by choosing the zeroth branch of the logarithm. 

A normal form $a \in \NF$ will be called an {\em algebraic path} (resp.~{\em rational path}) of length $n \ge 0$ if it has the form
$$
a = \theta_0\,  \tp_{\alpha_1} \theta_1\, \cdots\,   \tp_{\alpha_n} \theta_n, 
$$
where $\alpha_i$ are exponents in  $\Omega \cap \cQ$ (resp. $\alpha_i \in \Omega \cap \cZ$), and
$$\theta_0 \in \PSL, \quad \theta_n \in T_1\quad \text{Êand }\quad \theta_1,\ldots,\theta_{n-1} \in T_1 \setminus \{\id\}.$$  
We will say that $a$ is of {\em affine type }if $\theta_i$ is an affine map for each $0 \le i \le n-1$.  All paths  of length $n = 0$ are of affine type.

Notice that each path $g \in \NF$ can be decomposed  as 
$$
g = a_0\, \gamma_1\, a_1\, \cdots\, \gamma_m\, a_m, \quad m \ge 0
$$
where each $a_i$ is an algebraic path and each $\gamma_i$ is either $\te$, $\tl$ or a power germ $\tp_{\alpha_i}$ with an exponent 
$\alpha_i \in \Omega \setminus \cQ$.  This decomposition is unique if we further require that there are no subpaths of the form
$$
\gamma_i\, a_i\, \gamma_{i+1} = \te\, \ts_\alpha\, \tl.
$$
In other words, we assume that each subpath $\te\ts_\alpha\tl$ is grouped together into written as the power map $\tp_\alpha$.  

The above unique decomposition of $g$ will be called the {\em algebro-transcendental }decomposition.  Each $a_i$ is will be called a
{\em maximal algebraic subpath }of $g$. 

The natural number $m$ will be called the {\em height }of $g$ and noted $\het(g)$.  
Hence, normal forms of height zero correspond to algebraic paths.

Given symbols $\eta_1,\eta_2 \in \{\te,\tl,\tp\}$, we will say that the maximal algebraic subpath $a_i$ {\em lies in a  
$[\gamma,\eta]$-segment }, if $\gamma_{i} = \eta_1$ and $\gamma_{i+1} = \eta_2$.  
\begin{Example}\label{example-tamenf}
The path $g = \te\, \ts_{\sqrt{2}}\, \tl \, \tl\, \tt_1 \te\, \ts_2\, \tl\, \tt_1\,\ti\, \te\, \ti$ is a normal form with algebro-transcendental decomposition
$$
g = \underbrace{\vphantom{\tp_{\sqrt{2}}} \id}_{a_0}\;  \underbrace{\tp_{\sqrt{2}}}_{\gamma_1}\; \underbrace{\vphantom{\tp_{\sqrt{2}}}\id}_{a_1}\;  \underbrace{\vphantom{\tp_{\sqrt{2}}}\tl}_{\gamma_2}   \underbrace{\vphantom{\tp_{\sqrt{2}}}\tt_1 \tp_2 \tt_1\ti}_{a_2} \underbrace{\vphantom{\tp_{\sqrt{2}}}\te}_{\gamma_3} \; \underbrace{\vphantom{\tp_{\sqrt{2}}}\ti}_{a_3}\; 
$$
The maximal algebraic subpaths $a_1$ and $a_2$ lie in $[\tp,\tl]$ and $[\tl,\te]$ segments, respectively. 
Notice that maximal algebraic subpaths can be the identity, as it is the case of $a_0$ and $a_1$.  
\end{Example}
We will say that a normal form $g \in \NF$ is {\em tame }if
\begin{itemize}
\item[(i)] Either $\het(g) = 0$ and $g$ is an algebraic path of affine type. 
\item[(ii)] Or $\het(g) \ge 1$ and each maximal algebraic subpath lying in a segment of type
$$
[\te,\tl],\; [\tl,\te],\; [\tl,\tp],\; [\tp,\te],\;\text{ or  }\; [\tp,\tp]
$$
is of affine type.
\end{itemize}
For instance, the normal form of the previous example is tame. We shall denote by $\NF_\tame$ the subset of tame normal forms.  
\begin{MainTheorem}[Normal form in $\Gr_{\PSL,\Exp}$]
The mapping 
$$\vphi: \NF \rightarrow \Gr_{\PSL,\Exp}$$
is surjective.  Moreover, this mapping is injective when restricted to $\NF_\tame$.
\end{MainTheorem}
The study of non-tame normal forms puts into play some difficult problems concerning the study of finite coverings $\cPC \rightarrow \cPC$ with
imprimitive monodromy groups.  This issue is strongly related to the well-known Ritt's decomposability theorem \cite{MR1501189}, which fully describes 
{\em monoid }Êstructure of the polynomials under the composition operation.

The following example shows that we cannot expect the map $\rho:\NF \rightarrow \Gr_{\PSL,\Exp}$ to be bijective without further restrictions.
\begin{Example}\label{example-dihedralT2}
Each Chebyshev polynomial $T_n(x)$ lies in $\Gr_{\PSL,\Exp}$, as it can be defined by the identity
$$
T_n = \vphi\, \tp_n\, \vphi^{-1}
$$
where $\tp_n(z) = z^n$ and $\vphi(z) = z + 1/z$ is given explicitly by
$$
\vphi = \ts_{-1}\tt_2 \ti \tt_{-1/4}\tp_2 \tt_{-1/2}\ti \tt_1.
$$
On the other hand, we can also express $T_2(z) = z^2 - 2$ as $\tt_{-2} \tp_2$.  Hence, the relation
$$
T_2\Big( z  + \frac{1}{z}\Big) = z^2 + \frac{1}{z^2}
$$
is equivalent to say that the normal form 
$$
\ts_{-1} \tt_2 \ti \tt_{-1/4}\tp_2 \tt_{-1/2}\ti \tt_1 \tp_2 \tt_{-1} \ti \tt_{1/2} \tp_{1/2}\tt_{1/4}\ti \tt_{-2} \tp_{1/2} \tt_2$$
defines the identity germ.  Of course, this normal form is not tame. 
\end{Example}
\begin{Remarks}\label{remarks-nf}
(1) 
Some readers will probably notice the similarities between the above normal form and Britton's normal form for HNN-extensions (see e.g. \cite{MR1812024}, IV.2).  
Indeed, there is a general notion of HNN-extension for groupoids 
(\cite{MR2273730}, section 8.4.1) which generalizes the usual notion for groups (see subsection~\ref{subsect-HNN-ext}).  At a first view, one could expect to prove that $\Gr_{\PSL,\Exp}$ is isomorphic to the HNN-extension of 
$\Gr_{\PSL}$ with the exponential $\te$ acting as the {\em stable letter}, i.e.~somehow conjugating the subgroupoids  $\Gr_{H_0}$ and $\Gr_{H_1}$.

This cannot hold (at least not in such a na\"{\i}ve way).  In fact, consider the subgroupoids
$\LGr_0$ and $\LGr_1$ obtained by restricting the groupoids $\Gr_{H_0}$ and $\Gr_{H_1}$ to the subdomains $\cC$ and $\cC^*$,  respectively.
Then, the exponential map indeed defines a morphism of groupoids by
\begin{align*}
\Theta : \LGr_0 &\longrightarrow \LGr_1\\
h &\longmapsto \te\, h\, \te^{-1}
\end{align*}
where the germ $\te^{-1}$ is chosen in such a way that $\source(h) = \target(\te^{-1})$.  At the level of objects, this induces the mapping $\Theta : \cC \rightarrow \cC^*$, 
$\Theta(p) = \exp(p)$.  

However, $\Theta$ is not an isomorphism of groupoids, since 
it annihilates all germs $\tt_{2\pi i k}$, with $k \in \cZ$; and identifies each two points in $\cC$ which differ by an integer multiple of $2\pi i$.  

As a matter of fact, $\Theta$ establishes an isomorphism between the groupoid $\LGr_1$ and the {\em quotient groupoid} $\LGr_0/\Ker(\Theta)$, which is simply the groupoid with the object set $\cC/2\pi i \cZ$ and morphisms given by the action of 
$\{\tt_a : a \in \cC\}$ and $\ts_{-1}$ modulo $2\pi i \cZ$.
   
(2) It is easy to see that $\Gr_{\PSL,\Exp}$ coincides with $\Gr_{T,\Exp}$, i.e.~the groupoid generated only by the translations and the exponential.  Indeed, one easily constructs the subgroups $S$ and $W$ by defining 
$$\ts_\alpha =  \te \tt_{\ln_0(\alpha)} \tl, \quad\text{ and }\quad \ti = \te^2 \tt_{i\pi} \tl^{2}.$$
for all $\alpha \in \cC^*$.  The Normal form Theorem could be formulated solely in terms of paths in $\Gr_{T,\Exp}$.  However, this would lead to a much more complicated enunciation and to the loss of  the analogy with the theory of HNN-extensions. 

(3) Notice that $\Gr_{\PSL}$ has a natural Lie groupoid structure, which is inherited from \' etale groupoid structure of 
$\Gr(\cPC)$ (see \cite{MR2012261}, section 5.5). Some readers may be wondering which is the relation between 
$\Gr_{\PSL}$ and the so-called {\em semi-direct product Lie groupoid }
$$\PSL \ltimes \cPC$$ 
which is naturally defined by the action of $\PSL$ on $\cPC$ (see \cite{MR2012261}, section 5.1).  

One can show that $\PSL \ltimes \cPC$ and $\Gr_{\PSL}$ are isomorphic as groupoids, but not as Lie groupoids.  Indeed, the source fibers of $\Gr_\PSL$ (i.e.~the sets $\source^{-1}(p)$, $p \in \cPC$)  have a discrete topology while all source fibers of $\PSL \ltimes \cPC$ are manifolds diffeomorphic to $\PSL$.  

(4) Another interesting construction can be obtained by combining the groupoids $\Germ(\PSL)$, $\Germ(exp)$ and $\Germ(\wp,\wp^\prime)$, where
$$\wp \colon \cC/\Lambda \to \cPC$$
is the Weierstrass  function associated to a period lattice $\Lambda \subset \cC$.  In this case, the resulting groupoid $\Gr$ would contain a rich class of rational maps, so-called {\em finite quotients of affine maps }(see \cite{MR2348953}), i.e.~rational maps $f$ of degree two or more which fit into commutative diagrams of the form
$$
\begin{tikzpicture}
  \matrix (m) [matrix of math nodes,row sep=3em,column sep=5em,minimum width=2em,text height=1.5ex, text
depth=0.25ex]
  {
     \cC/\Lambda & \cC/\Lambda \\
     \cPC &  \cPC \\};
  \path[-stealth]
    (m-1-1) edge node [left] {\small $\Theta$} (m-2-1)
            edge  node [above] {\small $l$} (m-1-2)
    (m-2-1) edge node [above] {\small $f$} (m-2-2)
    (m-1-2) edge node [right] {\small $\Theta$} (m-2-2);
\end{tikzpicture}
$$
where $l(z) = az+b$ is an affine map defined on $\cC/\Lambda$ and  $\Theta: \cC/\Lambda \to \cPC$ is a finite covering.  For instance 
(see \cite{MR2193309}, Problem 7-f), for $\Lambda = \cZ \oplus i \cZ$ and $l(z) = (1+i) z$, the germ 
$\wp \, l \, \wp^{-1}$ is (up to a conjugation by a M\" oebius map) the quadratic rational map $h(z) = (z + 1/z)/2i$.
\end{Remarks}
\subsection{Powers and affine maps}\label{subsection-poweraffmaps}
As a consequence of the Main Theorem, we are going to obtain a generalization of a result of S. Cohen. Let $R$ be an arbitrary multiplicative subgroup of $\cC^*$ and let $\Pow_R$ be the set of germs determined by all the branches of the power maps
$$
\cC^* \ni z \mapsto z^r, \quad\text{ with }r \in R
$$
Clearly, the associated groupoid $\Germ(\Pow_R)$ is simply obtained by taking the union of $\Pow_R$ with the identity germs $\id$ at $0$ and $\infty$. 
As above, for each $r \in R$, we denote by 
$$\tp_r = \te \ts_r \tl$$ 
the germ of power map obtained by choosing the zeroth branch of the logarithm.

Initially motivated by a question of Friedman, several authors (cf. \cite{MR1184319}) considered the
groupoid 
$$
\Gr_{\Aff,\Pow_R} = \Germ(\Aff,\Pow_R)
$$
whose elements are obtained by finite compositions of germs of affine and power maps.  In particular, they studied the following property:
\begin{Definition}
We will say that $\Gr_{\Aff,\Pow_R}$ has the {\em amalgamated structure property} if 
each element $\Gr_{\Aff,\Pow_R}$ can be uniquely defined by a path
$$
[g_0, \, \tp_{r_1},\tt_{a_1}, \ldots, \tp_{r_{n}},\tt_{a_{n}}]
$$
for some $n \ge 0$, where $g_0  \in \Germ(\Aff)$, $r_i \in R\setminus\{1\}$ and $a_i \in \cC$ for $i = 1,\ldots,n$, such that $a_i$ is nonzero for $1 \le i \le  n-1$.  
\end{Definition}
In particular, this property implies that, given $n \ge 1$ and two sequences of constants $r_1,\ldots,r_n \in R\setminus\{1\}$ and $a_1,\ldots,a_n \in \cC$ 
with $a_i$ nonzero for $1 \le i \le n-1$,   the germ defined by
$$
z \mapsto (a_1 + (a_2 + \cdots + (a_n + x)^{r_n}\cdots )^{r_2})^{r_1}
$$
(where we choose arbitrary branches for the  power maps)  cannot be the identity.

Building upon a method originally introduced by White in \cite{MR969681},  Cohen proved in \cite{MR1326133} that 
$\Gr_{\Aff,\Pow_{\cQ_{>0}}}$ has the amalgamated structure property (i.e.~one takes $R$ equal to  $\cQ_{>0}$).

Using our normal form Theorem, we prove the following:
\begin{Theorem}\label{theorem-amalgamatedpoweraffine}
The groupoid $\Gr_{\Aff,\Pow_R}$ has the amalgamated structure property if  $R \cap \cQ_{< 0} = \emptyset$.
\end{Theorem}
Equivalently, we assume that for each $r \in R$, the ray $r \cQ_{<0}$ does not intersect $R$.
\begin{figure}[!h]
\centering
\begin{tikzpicture}[every node/.append style={font=\scriptsize}]

    \draw [->,black,line width=0.8pt] (-1.5,0) -- (1.5,0) ;
    \draw [->,black,line width=0.8pt] (0,-1.5) -- (0,1.5) ;
    \draw [black,line width=0.8pt] (0,0) -- (1,1) ;
    \draw [dashed,black,line width=1.2pt] (-1.5,-1.5) -- (0,0) ;
    \draw (1,1) node {$\bullet$};
   
     \node[right] at (1,1) {$r \in R$};
     \node[right] at (-1.9,-0.8) {$r\cQ_{<0}$};
\end{tikzpicture}
\label{hypothesisonR}
\end{figure}

\begin{Remark}
Assume $R$ is the multiplicative subgroup of $\cC^*$ generated by $\exp(2\pi i \lambda_1),\ldots,\exp(2 \pi i \lambda_n)$, for some collection of complex numbers $\lambda_1,\ldots,\lambda_n$.  Then, the condition $R \cap \cQ_{< 0} = \emptyset$ is equivalent to  the following {\em non-resonance condition}: 
$$
\Big( \frac{1}{2} + i\ln(\cQ_{>0})\Big) \, \cap \, \big( \cZ  + \lambda_1 \cZ + \cdots + \lambda_n \cZ\big) = \emptyset
$$
where $\ln$ denotes the principal branch of the logarithm function.
\end{Remark}
\subsection{Generalized Witt algebras}\label{subsect-willalgebras}
We describe another consequence of the Normal Form Theorem.   Let $\monoid$ be an additive sub-monoid of $\cC$ (i.e.~a subset $\monoid \subset \cC$ which is closed under addition and contains zero).  Following \cite{MR0396708}, we define the {\em generalized Witt algebra} $\vW(\monoid)$ as the $\cC$-vector space with a basis $\{w_g : g \in \monoid\}$, subject to the Lie multiplication
$$
[w_g,w_h] = (g - h) w_{g + h}.
$$
Each basis element can be represented by a (possibly multivalued) complex vector field on $\cPC$ given by
$$
w_g = z^g \left( z \frac{\partial}{\partial z} \right) 
$$
whose flow at time $a$ is given by the multivalued map
$$
\fexp(a\, w_g) = z \mapsto
\begin{cases}
 (- a g + z^{-g})^{-1/g},& \text{ if }g \ne 0\\
 \exp(a) z,& \text{ if }g = 0.
 \end{cases}
$$
Following the conventions of the first subsection, we are going to denote also by $\fexp(a\, w_g)$ the germs in $\Gr(\cPC)$ obtained by taking all possible determinations of the maps $z \mapsto (- a g + z^{-g})^{-1/g}$ at all points of its domain of definition.
\begin{Example}\label{example-witt}
For $\monoid = \cZ$ we obtain the classical Witt algebra $\vW(\cZ)$. The subalgebra $\vW(\cZ_{\le 0}) \subset \vW(\cZ)$ plays an important role in holomorphic dynamics. The flow maps  in this  subalgebra can be written as
$$
z  \mapsto \fexp(a\, w_{-k})(z) = z\, (1 - a k z^{-k})^{1/k}, \; \quad \forall k \in \cZ_{\ge 0}, \forall a \in \cC
$$
and they generate a well-known subgroup of the group $\Diff(\cC, \infty)$ of germs of holomorphic diffeomorphisms fixing the infinity. 
\end{Example}
\begin{Theorem}\label{theorem-NFinWitt}
Let $\monoid$ be an arbitrary additive sub-monoid of $\cC$.  Then, for all $n \ge 1$, all  scalars $a_1,\ldots,a_n \in \cC \setminus \{0\}$ and all elements $g_1,\ldots,g_n \in \monoid \setminus \{0\}$ such that
$g_{i+1}/g_i \notin \cQ_{<0} \cup \{1\}$, the germ
$$
z \mapsto \fexp(a_1 w_{g_1}) \cdots \fexp(a_n w_{g_n})(z)
$$ 
cannot be the identity. 
\end{Theorem}
\begin{Remark}\label{remark-freepoints}
 The condition $g_{i+1}/g_i \ne 1$ must be imposed due to the trivial relation
$$
\fexp(a w_{g})\fexp(b w_{g}) = \fexp((a+b) w_{g}) 
$$
Moreover, there are numerous counter-examples to the above result if drop the assumption $g_{i+1}/g_i \ne -1$.  For instance, given $a \in \cC^*$, consider the so-called {\em two parabolic group} $G_a \subset \PSL$, which is the group generated by the time $a$ flows maps of $w_{-1}$ and $w_1$, namely
$$
z \mapsto \fexp(a w_{-1})(z) = z + a, \qquad z \mapsto \fexp(a w_{1})(z) = \frac{z}{1 + az} 
$$
Following \cite{MR0258975},  we say that $a$ is a {\em free point} if $G_a$ is a free group.  There are plenty of non-free points. 
For instance, Ree showed in \cite{MR0142612} that the real segment  $]-2,2[$ is contained in an open set where the non-free points are densely distributed. 

Assume that $a$ is a non-free point.  Then, by definition, there exist a $n \ge 1$ and nonzero integers $p_1,q_1,\ldots,p_n,q_n$ such that
$$
\fexp(p_1 a w_{1}) \fexp(q_1 a w_{-1})\cdots \fexp(p_n a w_{1}) \fexp(q_n a w_{-1}) = \id
$$
Clearly, each relation of this type would give a counter-example to the above Theorem if the assumption $g_{i+1}/g_{i+1} \ne -1$ were dropped.  
\end{Remark}
Our next goal is to state a normal form result for the groupoid 
$$
\Gr_{\monoid} = \Germ\big(\{\fexp(a w_{g}) : g \in \monoid, a \in \cC\}\big)
$$
For this, given $g \in \monoid$, and $a \in \cC$, we define the following germ 
$$
\texp_{a,g} = 
\begin{cases}
\tp_{-1/g}\tt_{-ag}\tp_{-g},& \text{ if }g \ne 0\\
\ts_{\exp(a)},& \text{ if }g = 0.
 \end{cases}
$$
where $\tt_a$ and $\ts_\alpha$ are the translation and scalings germs, respectively; and the power map $\tp_r$ is defined as in subsection~\ref{subsection-poweraffmaps}. In other words, $\texp_{a,g} \in \Gr_{\monoid} $ is simply the germ obtained from the (multivalued) flow map $\fexp(a w_g)$ by choosing the zeroth branch of the logarithm in the definition of the power maps $x \mapsto x^{-g}$ and $x \mapsto x^{-1/g}$. 

The phenomena described in the previous remark leads us to define the following concept.  We say that an additive sub-monoid $\monoid$ of $\cC$ has {\em no rational antipodal points} if 
$$\monoid \cap \big( \monoid \cQ_{\le 0} ) =  \{0\}.$$
Using our Main Theorem, we shall prove the following:
\begin{Theorem}[Normal form in $\Gr_{\monoid}$]\label{theorem-NFinmonoid}
Suppose that $\monoid$ has no rational antipodal points.  Then, each element of the groupoid 
$\Gr_{\monoid}$
is uniquely defined by a path
$$
[\texp_{a_0,0}, \texp_{a_1, g_1}, \ldots, \texp_{a_n, g_n}]
$$
for some $n \ge 0$, $g_i \in \monoid$ and $a_i \in \cC$  such that:
\begin{itemize}
\item[(i)] $g_i \in \monoid \setminus \{0\}$ and $a_i \in \cC \setminus \{0\}$, for $1 \le i \le n$,
\item[(ii)] $\Im(a_0)\in ]-\pi i, \pi i]$,
\item[(iii)] $g_i \ne g_{i+1}$, for $1 \le i \le n-1$.
\end{itemize}
\end{Theorem}
\begin{Remark}
For sub-monoids $\monoid$ having antipodal points, it follows from Remark~\ref{remark-freepoints} that a normal form 
result as above would depend on precise characterization of the set of free points.  This seems to be a very difficult problem. As a hint, 
we refer to figure~\ref{fig4}, reproduced from \cite{MR2369190}.  It shows a numerically computed representation the set of free points in the plane $\cC_\lambda$, where $\lambda = a^2/2$.  

\begin{figure}[htbp]
\begin{center}
{ \includegraphics[height=4cm]{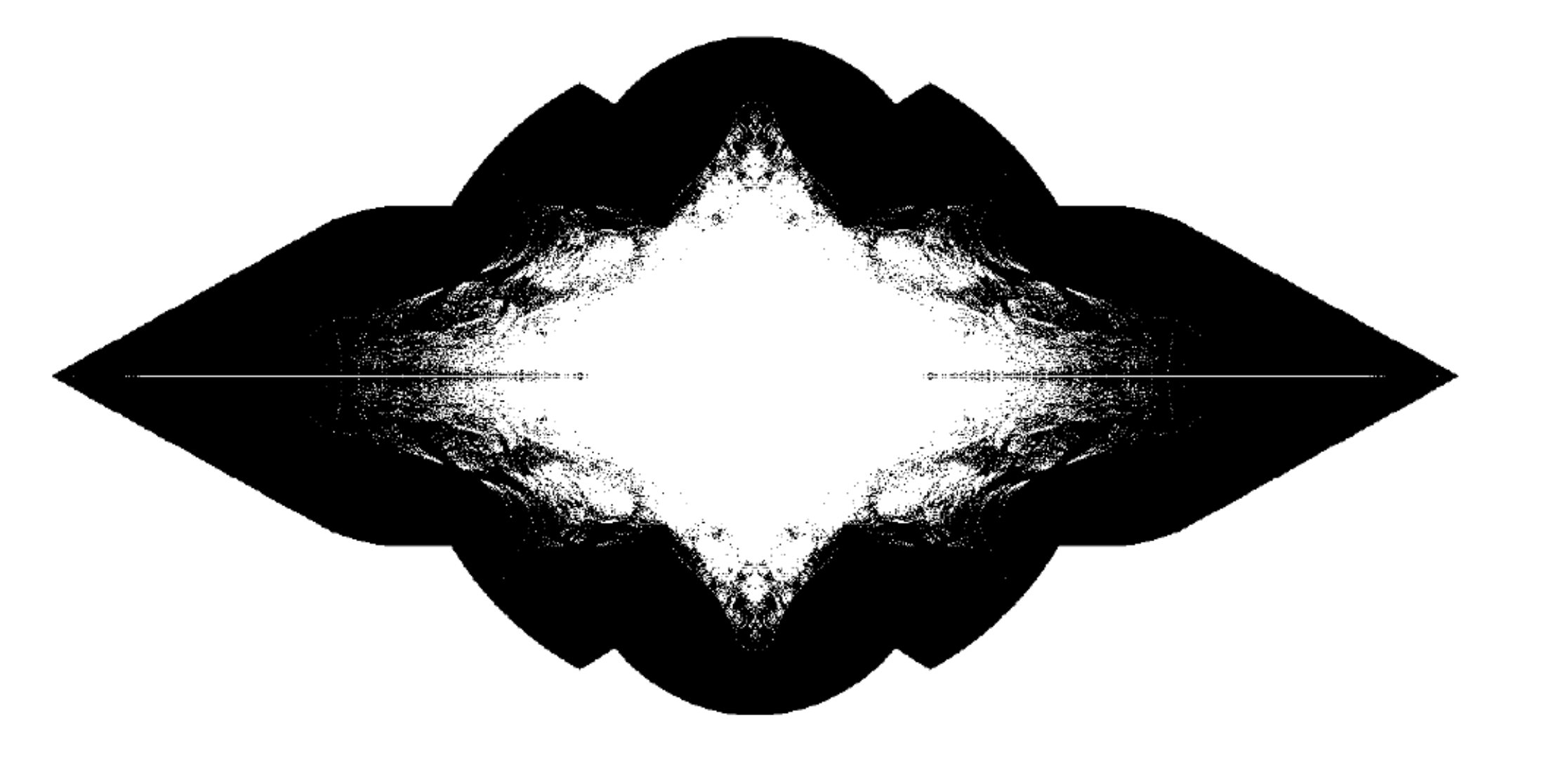}}
\caption{
The known free points are unshaded.}\label{fig4}
\end{center}
\end{figure}

\end{Remark}
\subsection{HNN-extensions in $\Homeo(\cR,+\infty)$}\label{subsect-HNN-ext}
Going in another direction, we can consider the analogous problem for the group $\Homeo(\cR,+\infty)$ of germs at $+\infty$ of 
real homeomorphisms defined in open intervals of the form $\{ x : x > x_0\}$ and which go to infinity as $x$ goes to infinity. 
The group operation being the composition.

Consider the following subgroups of $\Homeo(\cR,+\infty)$,
\begin{align*}
&T = \{ \tt_a \colon x \mapsto x + a\, , \, a \in \cR\}, \quad S^+ =  \{ \ts_\alpha \colon x \mapsto \alpha x\, ,\, \alpha \in \cR_{>0}\}\\
&\Exp = \{\te \colon x \mapsto \exp(x), \; \te^{-1}: x \mapsto \ln(x)\}, 
\end{align*}
where $\ln: \cR^* \rightarrow \cR$ is obviously taken as the real branch of the logarithm.  Let $\Aff^+ = T \rtimes S^+$ denote the subgroup of real positive affine maps.  
As it is well known, the conjugation by the exponential map defines an isomorphism 
\begin{align*}
\theta : T &\longrightarrow S^+ \\
 \tt_a &\longmapsto \te \tt_a \te^{-1} = \ts_{\exp(a)}
\end{align*} 
and we can consider the group $\Aff^+{\star_\theta}$ derived from $(T,\Aff^+,\theta)$ by HNN extension. We recall (see e.g.~\cite{MR1954121}, 1.4) that given a group $G$ with presentation $G = \langle F | R \rangle$ and an isomorphism $\theta: H \rightarrow K$ between two subgroups $H,K \subset G$, the {\em HNN extension derived from
$(H,G,\theta)$ }is a group $G{\star_\theta}$ with presentation
$$G{\star_\theta} = \big\langle F, \stable\  |\  R,\, \stable h \stable^{-1} = \theta(h), \forall h \in H \big \rangle$$
The new generator $\stable$ is called {\em stable letter}.

Consider now the subgroup $\gr_{\Aff^+,\Exp}$ of $\Homeo(\cR, +\infty)$ generated by $\Aff^+ \cup \Exp$.  From the universal property of the HNN extensions, we know that there is an unique morphism 
$$\phi: \Aff^+{\star_\theta} \longrightarrow \gr_{\Aff^+,\Exp}$$ 
which is the identity when restricted to $\Aff^+$ and which maps the stable letter to the exponential map.

We claim that $\gr_{\Aff^+,\Exp}$ contains no other relations besides the one expressing that $\exp$ conjugates $T$ to $S^+$.  In other words,
\begin{Theorem}\label{Theorem-realaffineexp}
$\phi: \Aff^+{\star_\theta} \rightarrow \gr_{\Aff^+,\Exp}$ is an isomorphism.
\end{Theorem}
\begin{Remark}
Based on the above result, we can obtain a quite economic presentation for the group $\gr_{\Aff^+,\Exp}$, namely
$$
\gr_{\Aff^+,\Exp} = \big\langle \cR, \stable\ \big | \ (\stable a \stable^{-1}) b (\stable a \stable^{-1})^{-1} = \exp(a) b,  \, \forall a,b \in \cR\big\rangle
$$
where $\cR$ is equipped with its usual additive group structure.  For instance, the multiplicative structure of $S^+$ is easily obtained by 
defining $\ts_{\exp(a)} := \stable a \stable^{-1}$.
\end{Remark}
As another consequence, we obtain a large collection of (apparently new) free subgroups inside $\Homeo(\cR, +\infty)$.  Indeed, consider the family of subgroups $\{T_n\}_{n \in \cZ^*} \subset \Homeo(\cR,+\infty)$ given by 
\begin{align*}
T_0 &= T,\quad & T_n &= \theta^n(T_0) = \te^n T_0 \te^{-n}, \qquad \forall n \in \cZ^*
\end{align*}
where, for $n > 0$ (resp.\ $n < 0$),  $\te^n$ denotes the $n$-fold composition of $\te$ (resp.  $\te^{-1}$).  Notice that $S^+ = T_1$.  We define 
$$
A_n = \theta^n(\Aff^+) = T_n \rtimes T_{n+1}, \quad \forall n \in \cZ
$$
where $T_{n+1}$ acts on $T_n$ by conjugation (exactly as $S^+$ acts on $T$).
\begin{Corollary}
The subgroup of $ \gr_{\Aff^+,\Exp}$ generated by $\bigcup_{n \in \cZ} A_n$ is isomorphic to the infinite free amalgamated product given by the following diagram
\begin{center}
\begin{tikzpicture}[description/.style={fill=white,inner sep=2pt}]
\matrix (m) [matrix of math nodes, row sep=2em,
column sep=0.8em, text height=1.5ex, text depth=0.25ex]
{  \cdots  &           & \Aff^+ &         & \Aff^+ &       & \Aff^+ &        & \cdots  \\
               & T       &            & T      &           & T    &           & T     &                    \\ };
\path[right hook->,font=\scriptsize]
(m-2-2) edge node[below right] {$ \mathrm{id} $} (m-1-3)
(m-2-4) edge node[below right] {$ \mathrm{id} $} (m-1-5)
(m-2-6) edge node[below right] {$ \mathrm{id} $} (m-1-7)
(m-2-8) edge node[below right] {$ \mathrm{id} $} (m-1-9);
\path[left hook->,font=\scriptsize]
(m-2-2) edge node[below left]  {$ \theta $} (m-1-1)
(m-2-4) edge node[below left] {$ \theta$} (m-1-3)
(m-2-6) edge node[below left] {$ \theta $} (m-1-5)
(m-2-8) edge node[below left] {$ \theta $} (m-1-7);
\end{tikzpicture}
\end{center}
where the north-east and north-west arrows are respectively the identity inclusions and the monomorphism $S^+ = \theta(T)$.
\end{Corollary}
In \cite{MR1184319}, Glass attributes to Higman the following question: 
\begin{center}
{\em Do $T$ and $\Pow^+ = \{x \mapsto x^r :  r \in \cR_{>0}\}$ generate their free product?}
\end{center}
The above Corollary allows us to answer this question affirmatively.  Indeed, as $A_1 = S^+ \rtimes \Pow^+$,
the above diagram shows that the subgroup $G_{\Aff^+,\Pow^+}$ of $\Homeo(\cR,+\infty)$ generated by $T\cup S^+ \cup \Pow^+$ has the presentation
$$
G _{\Aff^+,\Pow^+}= \left(T \rtimes S^+\right) \; \underset{S^+}\star\;  \left( S^+ \rtimes \Pow^+ \right)
$$
where the amalgam is obviously made over $S^+$. 
\subsection{Transseries and a finitary version of {\em Lemme 1}}\label{subsect-finitarylemme1}
We follow the notation from \cite{0909.1259}.  Let $\transs = \cR \ltriple x \rtriple$ be the real ordered field of well-based transseries and $\gtranss \subset \transs$ be the subset large positive transseries.  Then, $\gtranss$ is a group under the composition operation and there is a injective homomorphism 
$$T:\gr_{\Aff^+,\Exp} \rightarrow \gtranss$$ 
which associates to each element $g \in \gr_{\Aff^+,\Exp}$ its transseries at infinity. Indeed, each germ in $\gr_{\Aff^+,\Exp}$ defines  element in the Hardy field $H(\cR_{an,exp})$ (see e.g.\cite{MR1848569}), and therefore this homomorphism is a direct consequence of the embedding of 
$H(\cR_{an,exp})$ into $\transs$ (see \cite{MR1848569}, Corollary 3.12).

In this subsection, we shall be concerned with the following property (see e.g. \cite{MR1399559}, \cite{MR1133882}):
\begin{Definition}
Given an element $\phi \in \gtranss$ and a subgroup $ H \subset \gtranss$, we shall say that $H$ and $\phi$ are {\em immiscible} if the subgroup generated by 
$H \cup \phi H \phi^{-1}$ is isomorphic to the free product  $H \star H$.
\end{Definition}
For each integer $k \ge 1$, let $G_k \subset \gtranss$ denote the subgroup   real formal series at $+\infty$ which are tangent to identity to order $k$, i.e.~the group of transseries of the form
$$
x \mapsto x + \sum_{j \ge k-1} b_j x^{-j},\quad \text{ with }b_j \in \cR
$$
The following problem is stated in \cite{MR889742} (see also \cite{MR1399559}, section 1.4): 

\noindent{\bf Immiscibility problem: }{\em 
Prove that $G_2$ and $\phi$ are immiscible in the following cases:
$$
\phi : x \mapsto \exp(x), \quad \phi :  x \mapsto x + \ln(x),  \quad\text{or}\quad \phi :  x \mapsto x^\lambda, 
$$
where $\lambda \in \cR_{>0} \setminus \cQ_{>0}$.}
\begin{Remarks}
(1) The immiscibility  problem naturally appears in the study of the Poincar\' e first return map in the vicinity of an elementary polycycle.   Such study is an essential ingredient in the proofs of the Finiteness Theorem of limit cycles for polynomial vector fields in the plane (see  in \cite{MR1399559} and \cite{MR1133882}).   According to the strategy sketched in  \cite{MR889742} and 
\cite{MR888240}, one expects that a positive answer to the immiscibility problem would allow to significantly simplify these proofs.\\
(2) 
The immiscibility problem has an obvious negative answer if  $G_2$ is replaced by $G_1$.  Indeed, given an arbitrary non-identity element
$f \in G_1$ and an scalar 
$a \in \cR^*$, consider the series
$$
g = \ts_{\exp(a)} f \ts_{\exp(-a)}
$$
which is also an element of $G_1$.  Then, using the identity $\te \tt_a = \ts_{\exp(a)} \te$ one can rewrite 
$$
g = \te \tt_a \te^{-1} f \te \tt_{-a} \te^{-1}
$$
Since the translation $\tt_a$ is an element of $G_1$, one obtains the following relation in the subgroup generated by $G_1 \cup \te G_1 \te^{-1}$:
 $$
\te \tt_a \te^{-1} f \te \tt_{-a} \te^{-1}  =  \ts_{\exp(a)} f \ts_{\exp(-a)},
 $$
which shows that this subgroup is not isomorphic to the free product $G_1 \star G_1$.
 \end{Remarks}

Notice that an element $f \in G_k$ can be written as the limit of a (Krull convergent) sequence $\{f_n\}_{n \ge k} \subset \gtranss$ given by
$$
f_0  = \id, \quad f_n = T\big( \fexp(a_n w_{-n})\big)\, f_{n-1}, \quad \text{ where }w_{-n} = x^{-n} \left( x \frac{\partial}{\partial x} \right)
$$
with constants $a_n \in \cR$ uniquely determined by $f$ and the flow maps  $\fexp(a_n w_{-n})$ being given by Example~\ref{example-witt}.
 
This motivates us to consider the subgroup $G_{k,\fin} \subset G_k$ of those elements $f$ which can be expressed as {\em finite words}, namely
$$
f = T\Big( \fexp(a_1 w_{-k_1}) \cdots  \fexp(a_n w_{-k_n}) \Big)
$$
for some $n \ge 0$, $a_i \in \cR$ and $k_i \in \cZ_{\ge k}$.  Notice that each $G_{k,\fin}$ is indeed defined by an analytic function in a neighborhood of infinity and lies in the image of the morphism $T$ considered above. It also lies in the Hardy field $H(\cR_{\an,\exp})$ (cf. \cite{MR1848569}, section 3.11).

In order to formulate our next result, let $\Paths$ be the subset of all non-identity elements $g \in \gr_{\Aff^+,\Exp}$ of the form
$$
g = g_1 \cdots g_n
$$
for some $n \ge 1$ and $g_i \in \{\te,\tl\} \cup \{\tp_r : r \in \cR_{>0} \setminus \cQ_{>0}\}$.  As a consequence of the previous Theorem and the Normal form Theorem, we obtain the following {\em finitary}  immiscibility property:
\begin{Theorem}\label{theorem-lemme1}
Let $g$ be an arbitrary element in  $\Paths$.  Then, $G_{2,\fin}$ and $\phi = T(g)$ are immiscible.
\end{Theorem}
\begin{Remarks}
(1) Notice that the result includes the case where $g$ is given by towers of exponentials and powers, such as
$$
g : x \rightarrow e^{e^{x^r}}, \quad r \in \cR_{>0} \setminus \cQ_{>0}
$$
However, it does not include the so-called {\em inverse log-Lambert map}, 
$$L: x \mapsto x + \ln(x)$$  
which is a solution of the differential equation 
$$\Big( \frac{x}{1 + x}\, \frac{d}{d x} \Big) L  = 1.$$ 
The map $L$ plays an important role in proof of the finiteness of limit cycles. Indeed, it constitutes one of the {\em building blocks} in the construction of the Dulac transition map near a hyperbolic saddle or a saddle-node. We believe that it is possible to adapt our proof to include this function in the statement of the above Theorem. \\
(2) The passage from $G_{2,\fin}$ to $G_2$ in the immiscibility problem seems to be outside the reach of the tools developed in this paper.  A possible strategy of proof could consist in appropriately identifying $G_2$ to some subset of {\em ends} in the Bass-Serre tree defined by the HNN-extension 
$\Aff^+{\star_\theta}$.
\end{Remarks}
\subsection{Acknowledgements}  
I would like to thank Dominique Cerveau, Robert Roussarie, Robert Moussu, Jean-Jacques Risler, Jean-Philippe Rolin, Jean-Marie Lion, Bernard Teissier, Norbert A'Campo, Frank Loray, \' Etienne Ghys  and Thomas Delzant for numerous enlightening discussions.
\section{Formal Theory in $\Gr_{\PSL,\Exp}$}\label{section-formaltheory}
In this section, we will start our proof of the Main Theorem.  As a first step, we recall some basic universal constructions in groupoid theory, following closely \cite{MR2273730}.
\subsection{Free product and quotient of groupoids}\label{subsect-freequotgroupoid}
Let $\Gr$ and $\HGr$ be groupoids, and let $j_1: \Gr \rightarrow K$, $j_2: \HGr \rightarrow K$ be morphisms of groupoids.  We say that these morphisms present $K$ as the {\em free product} of $\Gr$ and $\HGr$ if the following universal property is satisfied:  if $g : \Gr \rightarrow \LGr$ and $h: \HGr \rightarrow \LGr$ are morphisms of groupoids which agree on $\Obj(\Gr) \cap \Obj(\HGr)$ then there is a unique morphism $k: K \rightarrow \LGr$ such that $kj_1 = g$, $kj_2 = h$.  Such free product always exists (see \cite{MR2273730}, section 8.1) and will be noted $\Gr\, *\, \HGr$.

If the groupoids $\Gr$ and $\HGr$ have no common morphism except the identity, the elements of $\Gr\, *\, \HGr$ are can be identified with the set of paths 
$$
[g_1, g_2, \ldots, g_n]
$$
which are either equal to $[\id]$ or where each $g_i$ belongs to either $\Gr$ or $\HGr$, no $g_i$ is the identity, and $g_i$, $g_{i+1}$ do not belong to the same groupoid.

We now recall the construction of the quotient of a groupoid by a set of relations.  In a groupoid $\Gr$, suppose given, for each object $p$, a set $R(p)$ of elements of $\Gr(p)$ (the vertex group at $p$).  The disjoint union $R$ of the $R(p)$ is called a set of {\em relations} in $\Gr$.  We define the {\em normal closure}
$\NGr = \NGr(R)$ of $R$ as the following subgroupoid:  Given an object $x \in \Obj(\Gr)$, a {\em consequence} of $R$ at $x$ is either the identity at $x$ or any morphism of the form
$$
a_n^{-1} \rho_n a_n \cdots a_1^{-1} \rho_1 a_1
$$
for which $a_i \in \Gr(x,x_i)$ and $\rho_i$, or $\rho_i^{-1}$, is an element of $R(x_i)$.  The set of all consequences at a point $x$, which we note $\NGr(x)$, is a subgroup of $\Gr(x)$ and the disjoint union $\NGr$ of all $\NGr(x)$ has the structure of a totally disconnected normal subgroupoid of $G$ (see \cite{MR2273730}, section 8.3), where by {\em totally disconnected groupoid} we mean a groupoid where each morphism have its source equal to its target.   It can be shown that $\NGr$ is the smallest wide normal groupoid of $\Gr$ which contains $R$. 

Let $\Gr/\NGr(R)$ be the quotient groupoid (see \cite{MR2273730}, Theorem 8.3.1).  The projection $\pi : \Gr \rightarrow \Gr/\NGr(R)$ has the following universal property: for each morphism of groupoids $f : \Gr \rightarrow \HGr$ which annihilates $R$, there exists a unique morphism 
$f^\prime : \Gr/\NGr(R) \rightarrow \HGr$ such that $f = f^\prime \pi$.
\subsection{Product normal form in $\Gr_{\PSL}\, *\, \PGr_{\Exp}$}\label{subsect-productnf}
The essence of our Normal Form Theorem is to present $\Gr_{\PSL,\Exp}$ as the quotient of a free product of groupoids by some explicit set of relations.  
For this, we consider the groupoids
$$
\Gr_{\PSL} = \Germ(\PSL), \quad\ \Gr_{\Exp} = \Germ(\{\exp\})
$$
and let $\PGr_{\Exp}$ denote the free groupoid on the graph of $\Gr_{\Exp}$, i.e.~the groupoid defined by the set of reduced paths on $\Gr_{\Exp}$ (see subsection~\ref{subsect-intronormalforms}). Let  $\Free = \Gr_{\PSL}\, *\, \PGr_{\Exp}$ be the free product of these groupoids.   

A first necessary step to obtain a normal form in $\vF$ is to describe the normal forms in $\Gr_{\PSL}$ and $\PGr_{\Exp}$.  We need two preparatory Lemmas:
\begin{Lemma}\label{lemma-nfPSL}
Each element $g \in \PSL$ can be written as one of the following expressions
$$
g = \ts_\alpha \tt_a \ti \tt_b \qquad \mbox{ or }\quad g = \ts_{\alpha} \tt_b
$$
for some uniquely determined constants $\alpha \in \cC^*$ and $a,b \in \cC$.  Moreover, if we consider the region $\Omega \subset \cC$ given by
$$\Omega =  \{\alpha : \Re(\alpha) > 0\} \cup \{\alpha : \Re(\alpha) = 0, \Im(\alpha) > 0\},$$
(see Figure~\ref{regionOmega}, at the Introduction), the following holds:
\begin{itemize}
\item[(i)] Each right coset of $H_0 = T \rtimes \{\ts_{-1}\}$ in $\PSL$  contains an unique element of the form
$$
g = \ts_{\rho} \ti \tt_b, \qquad \mbox{ or }\quad g = \ts_{\rho}
$$
for some constants $b \in \cC$ and $\rho \in \Omega$.
\item[(ii)] Each right coset of $H_1 = S \rtimes \{\ti\}$ in $\PSL$ contains an unique element of the form
$$
g = \tt_a \ti \tt_b, \qquad g = \tt_c \ti, \qquad \mbox{ or }\quad g = \tt_{b}
$$
for some  constants $c \in \cC \setminus \{0\}$, $b \in \cC$ and $a \in \Omega$ such that $b \ne -1/a$.
\end{itemize}
\end{Lemma}
\begin{proof}
The first part of the Lemma follows from the well-known presentation of $\PSL$ (see e.g.\cite{MR803508}, XI,\textsection 2).  In particular, we recall the following relation in $\PSL$,
$$\frac{1}{a + \frac{1}{z}} = -\frac{1}{a^2}\left(-a + \frac{1}{z + \frac{1}{a}} \right), \quad \forall a \in \cC \setminus \{0\}, z \in \cC,$$
or, equivalently, 
$$
 \ti \tt_a \ti = \ts_{-1/a^2} \tt_{-a} \ti \tt_{1/a}. 
$$
Now, in order to prove items (i) and (ii), it suffices to study the orbit of $\ts_{\alpha} \tt_b$ and $\ts_\alpha \tt_a \ti \tt_b$ under 
the left multiplication by $H_0$ and $H_1$, respectively.  

For instance, given $\ts_\alpha \tt_a \ti \tt_b \in \PSL$ such that $a \ne 0$, the above relation in $\PSL$ allows us to write
$$
\ts_\alpha \tt_a \ti \tt_b \equiv (\ts_{-a^2}\ts_\alpha \ti) \, \ts_\alpha \tt_a \ti \tt_b \equiv \ts_{-a^2}\ti \tt_a \ti \tt_b \equiv \tt_{-a} \ti \tt_{b + 1/a} 
$$
where $\equiv$ denotes the equivalence in $H_1\backslash \PSL$.  Therefore,  the coset $H_1 \ts_\alpha \tt_a \ti \tt_b$ contains either an element of the form $\tt_a \ti \tt_b$ with $a \in \Omega$ and $b \ne -1/a$ or an element of the form $\tt_c \ti$, with $c \ne 0$.
\end{proof}
To state the next result, we introduce the symbols
$$
\tl_k : z \mapsto \ln_k(z), \quad \forall k \in \cZ
$$
where, we recall, $\ln_k$ denotes the $k^{th}$ branch of the logarithm.  Notice that $\tl_0 = \tl$.
\begin{Lemma}\label{lemma-nfEXP}
Each element in $\PGr_{\Exp}$ is either the identity $\id$ or  a  path of the form
$$
[\tl_{k_1}, \ldots,  \tl_{k_n}, \underbrace{\te, \cdots, \te}_{\text{s-times}}]
$$
for some positive integers $n, s$, not both zero, and integers $k_1,\ldots,k_n \in \cZ$ such that the rightmost germ $\tl_{k_n}$ and the leftmost germ $\te$ in the path are not mutually inverses.
\end{Lemma}
\begin{proof}
Each germ $g \in \Gr_{\Exp}$ is defined by a path
$$
[g_1,g_2,\ldots,g_m]
$$
where each $g_i$ is either equal to $\te$ or to $\tl_k$ for some $k \in \cZ$.  We transform this path to a reduced one by successively canceling out each two consecutive germs $g_i,g_{i+1}$ such that $g_i g_{i+1} = \id$.  

Recall now the following (unique) two relations 
in $\Gr_{\Exp}$ (see the discussion at the Introduction),
$$(1)\; \te \tl_k = \id\quad \text{ and }\quad  (2)\; \tl_k \te = \id,\; \text{if $\source(\te) \in J_k$},$$ 
for all $k \in \cZ$.  Therefore, after performing all possible cancellations in the above path, we either obtain the identity path, or  a path as above such that no germ $\tl_k$ has a germ $\te$ to its left; and furthermore, that no consecutive germs $\tl_k, \te$ are mutually inverse.  This is precisely a path of the form in the statement of the Lemma.
\end{proof}
We now consider normal forms inside the free product groupoid $\Free = \Gr_{\PSL}\, *\, \PGr_{\Exp}$.
\begin{Definition}\label{def-productnf}
A {\em product normal form}  in $\Free$ is a path of the form 
$$
g  =  [g_0, h_1, f_1,g_1, \ldots, h_n, f_n,g_n]
$$
for some $n \ge 0$, such that the following holds:
\begin{itemize}
\item The germ $g_0$ lies in $\Gr_\PSL$ (with possibly $g_0 = \id$).
\item For $1 \le i \le n$, $h_i$ is either equal to  $\te$ or to $\tl_k$, for some $k \in \cZ$.
\item If $h_i = \te$ then $f_i \in H_0$ and $g_i$ is given by item (i) of Lemma~\ref{lemma-nfPSL}. 
\item If $h_i = \tl_k$ then $f_i \in H_1$ and $g_i$ is given by item (ii) of Lemma~\ref{lemma-nfPSL}. 
\item There are no subpaths $[\tl_k,\id,\id,\te]$ or $[\te,\id,\id,\tl_k]$ such that the germs $\tl_k$ and $\te$ are mutually inverse.
\end{itemize}
We denote by $\PNF$ the set of all product normal forms.  
\end{Definition}
As a consequence of the definition of $\Free$ and the previous two Lemmas, we obtain the following
\begin{Proposition}
Each morphism of $\Free$ can be uniquely defined by an element of $\PNF$.  
\end{Proposition}
\begin{proof}
By the definition of a free product, each non-identity element $g \in \Free$ can be uniquely identified with a path
$$
g = [g_1, g_2,  \ldots, g_n]
$$
such that the following conditions hold:
\begin{itemize}
\item $g_i \in \Gr_{\PSL} \cup \PGr_{\Exp}$, for $i = 1,\ldots,n$;
\item no two consecutive morphisms $g_i$, $g_{i+1}$ belong to the same groupoid.
\item No $g_i$ is the identity morphism. 
\end{itemize}
Given such a path, we can uniquely obtain a path in $\PNF$.  Indeed, proceeding from left to right, for $i = 1,\ldots,n$, we do the following:
\begin{itemize}
\item[(1)] If $g_i \in \PGr_{\Exp}$ then, we use Lemma \ref{lemma-nfEXP} to write
$$
g_i = [\tl_{k_1},\ldots,\tl_{k_n},\te,\ldots,\te]
$$
and, in  the expression of $g$, we replace $g_i$ by the subpath
$$[\tl_{k_1},\id,\id,\tl_{k_2},\ldots,\tl_{k_n},\id,\id,\te,\id,\id,\ldots,\te]$$
\item[(2)] If $i \ge 1$, $g_i$ belongs to $\Gr_{\PSL}$  and $g_{i-1}$ has a $\te$ as its last symbol then we use Lemma \ref{lemma-nfPSL} to write
$$g_i = f\, g^\prime, \text{ for some }f \in H_0, \text{ and $g^\prime$ given by Lemma \ref{lemma-nfPSL}.(i)} $$
and we replace $g_i$ by the subpath $[f,g^\prime]$ in  the expression of $g$.
\item[(3)] If $i \ge 1$, $g_i$ belongs to $\Gr_{\PSL}$ and $g_{i-1}$ has a $\tl_k$ as its last symbol then we use Lemma \ref{lemma-nfPSL} to write
$$g_i = f\, g^\prime, \text{ for some }f \in H_1,  \text{ and $g^\prime$ given by Lemma \ref{lemma-nfPSL}.(ii)} $$
and we replace $g_i$ by the subpath $[f,g^\prime]$ in  the expression of $g$.
\end{itemize}
This concludes the proof.
\end{proof}
\begin{Remark}
Recall that the subgroup of translations by $2\pi i \cZ$ lies in the intersection 
$\Gr_{\Exp}\cap\Gr_\PSL$.  Therefore, the normal forms in the free product groupoid $\Gr_{\PSL}\, *\, \Gr_{\Exp}$ are more subtle to describe than those in 
$\Free$.  
\end{Remark}
Let now $\NF$ be the set of normal form paths defined in Remark~\ref{remark-defNF} of the Introduction. 
Clearly, there is a natural embedding of $\NF$ into $\PNF$ given by
$$[g_0,h_1,g_1,\ldots,h_n,g_n] \in \NF \longrightarrow  [g_0,h_1,\id,g_1,\ldots,h_n,\id,g_n]  \in  \PNF$$
To simplify the notation, we will keep the symbol $\NF$ to denote the image of this embedding. 
\subsection{Quotienting $\Gr_{\PSL} \, * \, \PGr_{\Exp}$}\label{subsect-listrel}
Now, we consider the following collection $\Rel$ of {\em relations} in $\vF$
$$
\Rel\;\;
\begin{cases}
\;\tl \te  = \tt_{-2\pi i r}, & \text{ if }\source(\te) \in J_r.\\
\;\te \ts_{-1} = \ti \te \\
\;\te \tt_a = \ts_{\exp(a)} \te, & \forall a \in \cC. 
\end{cases}
$$
where $\tl = \tl_0$ is the $0^{th}$ branch of the logarithm.    Notice that, for simplicity, we have 
written these relations in the form of an equality of germs $w = u$, but this should be understood as 
saying that $w$ composed with the inverse of $u$  is a relation (in the sense of subsection \ref{subsect-freequotgroupoid}) at 
every point where the corresponding germs are defined.

Let $\vF/\NGr(\Rel)$ denote the quotient groupoid, as defined in the previous subsection, and let
$$\pi : \vF \longrightarrow \vF/\NGr(\Rel)$$ 
be the canonical morphism.  
The following Theorem will be proved in the next subsection. 
\begin{Theorem}\label{theorem-normalforminquotient}
Each element in the quotient $\Free/\NGr(\Rel)$ is uniquely defined by a normal form in $\NF$. 
\end{Theorem}
We now observe that, by construction and the universal property of $\Free$, there is an uniquely defined groupoid epimorphism 
$$\phi: \Free \rightarrow \Gr_{\PSL,\Exp}$$ 
which is induced by the inclusion morphisms $\Gr_{\PSL} \rightarrow \Gr_{\PSL,\Exp}$ and $\Gr_{\Exp} \rightarrow \Gr_{\PSL,\Exp}$. 

Using the obvious relations between the exponential, the affine maps and the involution, we conclude that this morphism factors out
through the canonical morphism $\pi : \Free \rightarrow \Free/\NGr(\Rel)$, i.e.~we have a commutative diagram
$$
\begin{tikzpicture}
  \matrix (m) [matrix of math nodes,row sep=3em,column sep=4em,minimum width=2em]
  {
     \Free & \Gr_{\PSL,\Exp} \\
     \Free/\NGr(\Rel) &  { }\\};
  \path[-stealth]
    (m-1-1) edge node [left] {$\pi$} (m-2-1)
            edge  node [above] {$\phi$} (m-1-2)
    (m-2-1) edge node [above] {$\vphi\;\;$} (m-1-2);
\end{tikzpicture}
$$
for an uniquely defined morphism $\vphi: \Free/\NGr(\Rel) \rightarrow \Gr_{\PSL,\Exp}$. 
As an immediate consequence of this discussion and Theorem~\ref{theorem-normalforminquotient}, we obtain
\begin{Corollary}
The first statement of the Main Theorem is true. 
\end{Corollary}
\subsection{Reduction to normal forms in $\Free/\NGr(\Rel)$}\label{subsect-confluentnf}
This subsection is devoted to the proof of Theorem~\ref{theorem-normalforminquotient}.  For this, we briefly recall the basic concepts of reduction systems (see e.g.~\cite{MR1629216}).  An {\em abstract reduction system} is a pair $(X,\rightarrow)$ where the {\em reduction}  $\rightarrow$  is a binary relation on the set $X$.  Traditionally, we write $x \rightarrow y$ (or $y \leftarrow x$) instead of $(x,y) \in \rightarrow$. The binary relation $\rightstar$ is the reflexive transitive closure of $\rightarrow$. In other words, $x \rightstar y$ if and only if there is $x_0,\ldots,x_n$  such that $x = x_0 \rightarrow x_1 \rightarrow \cdots \rightarrow x_n = y$.  The binary relation $\leftrightstar$ is the reflexive transitive symmetric closure of $\rightarrow$.  Equivalently, $x \leftrightstar y$ if and only if there are $z_1,\ldots,z_n \in X$ such that
$$
x \leftrightarrow z_1 \leftrightarrow z_2 \cdots \leftrightarrow z_n \leftrightarrow y
$$
where $\leftrightarrow = \leftarrow \cup \rightarrow$.
We also say that:
\begin{itemize}
\item $x \in X$ is {\em reducible} if there is a $y \in X$ such that $x \rightarrow y$.
\item $x \in X$ is in {\em normal form} if it is not reducible.
\item $x \in X$ is a {\em normal form} of $y \in X$ if $y \rightstar x$ and $x$ is a normal form.   
\item $x,y \in X$ are {\em joinable} if there is a $z \in X$ such that $x \rightstar z \leftstar y$.
\end{itemize}
A reduction system $(X,\rightarrow)$ is called {\em terminating} if there is no infinite descending chain $x_0 \rightarrow x_1 \rightarrow \cdots$.  In this case, each element $x$ has at least one normal form.  A reduction system $(X,\rightarrow)$ is called {\em confluent} if $y_1 \leftstar x \rightstar y_2$, implies that the elements $y_1$ and $y_2$ are joinable.  We say that $(X,\rightarrow)$ is {\em Church-Rosser} if $x \leftrightstar y$ implies that $x$ and $y$ are joinable.  These two properties are usually pictured by the following respective diagrams
$$
\begin{tikzpicture}
  \matrix (m) [matrix of math nodes,row sep=3em,column sep=4em,minimum width=2em]
  {
     x & y_1 \\
     y_2 &  z \\};
  \path[->]
    (m-1-1) edge node [left] {$*$} (m-2-1)
            edge  node [above] {$*$} (m-1-2);
  \path[-stealth,dashed]          
    (m-2-1) edge node [above] {$*$} (m-2-2)
    (m-1-2) edge node [right] {$*$} (m-2-2);
\end{tikzpicture}\quad
\begin{tikzpicture}
  \matrix (m) [matrix of math nodes,row sep=3em,column sep=4em,minimum width=2em]
  {
     x & & y \\
     & z &   \\};
  \path[<->]
    (m-1-1) edge node [above] {$*$} (m-1-3);
  \path[-stealth,dashed]          
    (m-1-1) edge node [left] {$*\, $ } (m-2-2)
    (m-1-3) edge node [right] { $\, *$} (m-2-2);
\end{tikzpicture}
$$
We shall use the following consequences of the definitions:
\begin{itemize}
\item[(i)] if $(X,\rightarrow)$ is terminating and confluent then every element has a unique normal form (see \cite{MR1629216}, Lemma 2.1.8).
\item[(ii)]  The Church-Rosser and the confluent properties are equivalent (see \cite{MR1629216}, Theorem 2.1.5).
\end{itemize}
We are going to apply this formalism to the set $X = \PNF$ of product normal forms (see definition~\ref{def-productnf}).  
In order to simplify the notation, in the remaining of this subsection, we shall identify a path $[f_1, \ldots, f_n]$ 
with a word $f_1 f_2 \cdots f_n$ in the letters $f_1,\ldots,f_n$.  We stress that the formal word $f_1 f_2 \cdots f_n$ should not be confounded with the element of the groupoid $\Gr_{\PSL,\Exp}$ defined by the corresponding path.
The letter $\tl_0$ will be written simply $\tl$. Moreover, the identity path $[\id]$ will be identified with the empty word $\eps$. Thus, for instance, the path $[\te,\id,\id,\te]$ will be written simply as $\te\te$.

First of all, we introduce the following {\em reduction rules} (recall that both sides of the $\Rightarrow$ relation should be seen as paths in $\Free$):
\begin{align*}
\tl_k &\Rightarrow \tt_{2\pi i k} \tl, & \forall k \in \cZ^*\\
\te \tl  &\Rightarrow \eps, \\
\tl \te  &\Rightarrow \tt_{-2\pi i r}, & \text{ if }\source(\te) \in J_r\\
\te \ts_{-1} &\Rightarrow \ti \te, \\
\te \tt_a &\Rightarrow \ts_{\exp(a)} \te, & \forall a \in \cC. \\
\tl \ti &\Rightarrow \ts_{-1} \tl &  \text{ if }\arg_0(\source(\ti)) \ne \pi\\
\tl\ti &\Rightarrow \tt_{2\pi i}\ts_{-1} \tl& \text{ if }\arg_0(\source(\ti)) = \pi\\
\tl \ts_\alpha &\Rightarrow \tt_b \tl,  , & \forall \alpha \in \cC^*\\
\end{align*}
where, in this last rule, we define $b \in \cC$ as follows:
$$
b = \begin{cases}
\ln_0(\alpha), & \text{ if $-\pi < \mathrm{arg}_0(\alpha) + \arg_0(\source(\ts_\alpha)) \le \pi$}\\
\ln_{-1}(\alpha), & \text{ if $\pi < \mathrm{arg}_0(\alpha) + \arg_0(\source(\ts_\alpha)) \le 2\pi$}\\
\ln_{1}(\alpha), & \text{ if $-2\pi < \mathrm{arg}_0(\alpha) + \arg_0(\source(\ts_\alpha)) \le -\pi$}\\
\end{cases}
$$
The reduction system $(\PNF,\rightarrow)$ is now defined as follows: Given $g,h \in \PNF$, we say that $g \rightarrow h$ if there exists some reduction rule $u \Rightarrow v$ as above such that
one can write 
$$
g = g^\prime u g^{\prime\prime}, \quad \text{ for some }g^\prime,g^{\prime\prime} \in \PNF,
$$
and $h \in \PNF$ is the product normal form of the path $g^\prime v g^{\prime\prime}$.
\begin{Remark}
Notice that the simple substitution $g^\prime u g^{\prime\prime} \rightarrow  g^\prime v g^{\prime\prime}$ would {\em not} map $\PNF$ into itself.  For instance, if $b = \ln_0(-2)$ then 
applying the fifth substitution rule to $g = \te\tt_a \te \tt_b$, one would obtain $\te \tt_a \ts_{-2} \te$, which is not in $\PNF$, since $\tt_a \ts_{-2}$ should be decomposed in $H_0 T_0$ as $(\tt_a \ts_{-1}) \ts_2$.
\end{Remark}
\begin{Proposition}\label{prop-pnfconfluent}
The reduction system $(\PNF,\rightarrow)$ is terminating and confluent.  Moreover, the set of normal forms of $(\PNF,\rightarrow)$ is 
precisely the subset $\NF$.
\end{Proposition}
\begin{proof}
We claim that there can be no infinite sequence of reductions.  To prove this, we define a well-order on $\PNF$ which will decrease after each reduction.

First of all, recall that a germ 
$f \in H_0 \cup H_1$ is either the identity or can be uniquely expressed as follows:
\begin{itemize}
\item[(a)] In $H_0$:  $f= \tt_a \ts_{-1}$, $f = \tt_a$ or $f = \ts_{-1}$, for some $a \in \cC \setminus \{0\}$,
\item[(b)] In $H_1$:  $f= \ts_\alpha \ti$, $f = \ts_\alpha$, or $f = \ti$, for some $\alpha \in \cC^* \setminus \{1\}$.
\end{itemize}
Accordingly, we define the {\em $h$-length} $l_h(f) \in \{0,1,2\}$ by
$$l_h(f) = 2\ \text{ if }f \in \{\tt_a \ts_{-1}, \ts_\alpha \ti\}, \;\; l_h(f) = 1\ \text{ if }f \in \{\tt_a, \ts_\alpha,\ts_{-1}, \ti\},$$
and we put $l_h(f) = 0$ if $f = \id$.

Consider now a path $g = g_0 h_1 f_1 g_1 \cdots h_n f_n g_n$ in $\PNF$.  We define, its {\em $h$-length} as the integer $n$-vector 
$$l_h(g) = (l_h(f_n), l_h(f_{n-1}), \ldots, l_h(f_1)) \in \{0,1,2\}^n$$
We further define $m(g)$ to be the total number of germs of type $\tl_k$, for $k \in \cZ^*$, and $n(g)$ to be the total number of germs of type $\te$ or $\tl$ in the expression of $g$.  

Finally, we define a total order in $\PNF$ by saying that 
$g < g^\prime$ if 
$$(m(g),n(g),l_h(g)) <_\lex (m(g^\prime),n(g^\prime),l_n(g^\prime)),$$
where $<_\lex$ is the lexicographical ordering in the set of positive integer vectors.

By inspecting the rules in $\Rel$, one sees that if $g^\prime \rightarrow g$ then $g < g^\prime$.  
Moreover, a path $g \in \PNF$ is not reducible if and only if the following holds
\begin{itemize}
\item $m(g) = 0$,
\item $l_h(g) = (0,\ldots,0)$ and, 
\item $g$ contains no subpath of the form $\te \tl$ or $\tl\te$. 
\end{itemize}
According to Definition~\ref{def-nf}, this corresponds precisely to say that $g \in \NF$.  Thus, we have proved that $(\PNF,\rightarrow)$ is terminating and that its set of normal forms is $\NF$.

In order to prove the confluence of the reduction system, we use Bendix-Knuth criteria as stated in \cite{MR1161694}, Lemma 6.2.4.  
Thus it suffices to consider all shortest paths for which at least two of the above reduction rules can be applied (i.e. they overlap) and 
show that the paths obtained after applying these reductions are then joinable.  For instance, one sees that
$$
\begin{tikzpicture}
  \matrix (m) [matrix of math nodes,row sep=3em,column sep=7em,minimum width=2em,text height=1.5ex, text
depth=0.25ex]
  {
     \te\tl\te & \te \tt_{-2\pi i r} \\
     \te &  \te \\};
  \path[-stealth]
    (m-1-1) edge node [left] {\tiny $\te\tl \rightarrow \eps$} (m-2-1)
            edge  node [above] {\tiny $\tl\te \rightarrow \tt_{-2\pi i r}$} (m-1-2)
    (m-2-1) edge node [above] {$*$} (m-2-2)
    (m-1-2) edge node [right] {$*$} (m-2-2);
\end{tikzpicture}
$$
The computation for the other possible overlaps is straightforward but quite tedious.  We omit this computation.
\end{proof}
We are now ready to prove Theorem \ref{theorem-normalforminquotient}:
\begin{proof}[Proof of Theorem \ref{theorem-normalforminquotient}]  We need to prove that each coset of $\Free/\NGr(Rel)$ contains exactly one element of $\NF$.
By the fact that $(\PNF,\rightarrow)$ is terminating, we know that each coset of $\Free/\NGr(Rel)$ contains at least one element of $\NF$.

Now, the essential remark is that the equivalence relation $\leftrightstar$ on $\PNF$ defines precisely the cosets of the quotient groupoid $\Free/\NGr(Rel)$.  Indeed, for each relation $u=v$ in the list $\Rel$ given at subsection~\ref{subsect-listrel}, one sees that $u \leftrightstar v$.  Reciprocally, for each reduction rule $u \Rightarrow v$, one sees that 
$u v^{-1}$ belongs to $\NGr(\Rel)$.

Therefore, assume that there exist two elements 
$g,g^\prime$ in $\NF$ such that $\pi(g) = \pi(g^\prime)$ (where $\pi:\Free \rightarrow \Free/\NGr(Rel)$ is the quotient map).  
This is equivalent to say that $g \leftrightstar g^\prime$.  Since $(\PNF,\rightarrow)$ is confluent, it is Church-Rosser.  Therefore, $g \leftrightstar g^\prime$ implies that $g$ and $g^\prime$ are joinable. But since both $g$ and $g^\prime$ are normal forms (and hence not reducible), we conclude that $g = g^\prime$. 
\end{proof}
\begin{Remark}{\it (Word problem and decidability)\ }
One could ask if the reduction system $(\PNF,\rightarrow)$ would allow us to algorithmically solve the word problem
in $\Free/\NGr(Rel)$. Equivalently, one asks if, given an element $g \in \PNF$, there exists an algorithm to decide if 
$$
g \rightstar \id.
$$
Notice that the reduction rules in $(\PNF,\rightarrow)$ assume the existence of an {\em oracle} which, given a complex constant $\alpha \in \cC$, will
answer affirmatively or negatively to the question 
$$\text{Is $\alpha = 0\; $?}$$  
Even assuming that the constants appearing in the initial path $g$ are, say, rational numbers, this oracle will eventually need to test new constants which are exp-log expressions in these initial constants, such as
$$
  e^{e^{e^{2\log(3/4)}} + e^{- 3e^{e^{10}}}} - e^{e^{e^{2\ln{5}}}} - \ln(\ln(3/2)). 
$$
The existence of an algorithm for the above oracle is strongly related to the decidability of $(\cR,\exp)$ and the known algorithms 
assume Schanuel's conjecture.   

On the positive side, using the results of \cite{MR1487790} and \cite{MR2078675} 
(see also \cite{MR1098783}, section 2.1),  one can prove the following:
{\em Assuming Schanuel's conjecture, the word problem is decidable for the groupoid
$$\Free_\cQ/\NGr(Rel)$$
where $\Free_\cQ$ denotes the free product groupoid $\Gr_{\mathrm{PSL}(2,\cQ)}  *  \PGr_{exp}$.}
\end{Remark}
\section{From normal forms to field extensions}\label{subsect-algebrotransc}
Our present goal is to prove the second part of the Main Theorem.  Many of the following constructions will be carried out for arbitrary normal forms, not necessarily satisfying the tameness property.  We shall explicitly indicate the points where this assumption will be necessary.

Given a point $p \in \cPC$, we denote by $(\Merom,\partial)$ the differential field of meromorphic germs at $p$ equipped with the usual derivation $\partial = d/dz$ with respect to some arbitrary local coordinate $z$ at $p$ (with 
constant subfield $\mathrm{Const}(\partial) = \cC$).

Given a normal form $g \in \NF$, with source point $p = \source(g)$, our next goal is to construct a sequence of field extensions
in $\Merom$ which will encode the necessary information to study the identity
$$\vphi(g) = \mathrm{id}.$$
wher, we recall that $\vphi: \NF \to \Gr_{\PSL,\Exp}$ is the mapping which associates a germ in $\Gr_{\PSL,\Exp}$ to each path in $\NF$.
\subsection{Algebraic paths and Cohen field}\label{subsect-cohenfielddef}
In this subsection, we consider field extensions defined by algebraic paths.  Let $a \in \NF$ be an algebraic path of length $n \ge 0$.  Thus, we can uniquely write
$$
a = \theta_0\,  \tp_{\alpha_1} \theta_1\, \cdots\,   \tp_{\alpha_n} \theta_n, 
$$
where each exponent $\alpha_i$ lie in $\Omega \cap \cQ \setminus \{1\}$,  $\theta_0 \in \PSL$,  $\theta_i \in T_1 \setminus \{\id\}$, for $1 \le i \le n$ and $\theta_1,\ldots,\theta_{n-1}$ are not the identity.

We consider the sequence of {\em algebraic }field extensions in $\Merom$
$$
E_n \subset E_{n-1} \subset \cdots \subset E_0 = E
$$
inductively defined as follows.  Firstly, $E_n = \cC(x_n)$ is the field defined by the identity germ $x_n = \vphi(\id)$.  Then, for each $i = 0,\ldots,n-1$, we define 
$$
E_i = E_{i+1}(x_i)
$$
where $x_i$ is a germ of solution of the algebraic equation
$$
x_i^{v} = \theta(x_{i+1})^{\, u}
$$
where we have written $\theta = \theta_{i+1}$ and $\alpha_{i+1} = u/v$ for some co-prime integers $u,v$.  Here, the branch of the $v^{th}$-root is uniquely chosen accordingly to the source/target compatibility condition determined by $a$.
We will say that the resulting field
$$E = \cC(x_0,\ldots,x_n)$$ 
is the {\em Cohen field of $a$}.

In the seminal paper \cite{MR1326133}, Cohen has studied the Cohen field for algebraic normal forms of affine type. In what follows, we shall make essential 
use of the following immediate consequence of a result in \cite{MR1326133}.
\begin{Theorem}[cf.~\cite{MR1326133},Theorem 3.2]\label{theorem-ofcohen}
Assume that $a$ is an algebraic normal form of affine type and length $n \ge 0$, with associated Cohen field $E = \cC(x_0,\ldots,x_n)$. Then, we can write
$$E = \cC(x_0,x_n)$$ 
i.e.~$x_1,\ldots,x_{n-1}$ are rational functions of $x_0$ and $x_n$.  
Moreover, if $a$ is not a rational path {\em(\footnote{Recall (see the Introduction) that an algebraic path is called {\em rational }if each power map in its basic decomposition has a exponent in $\Omega \cap \cZ$.})} then $E$ is a strict algebraic extension of $\cC(x_n)$. 
\end{Theorem}
Notice that the second statement of the Theorem does not hold if $a$ is not of affine type.
\begin{Example}
Consider the algebraic normal form
$$
a = \tp_{1/2} \tt_{-1/4}\, \tp_2\, \ttheta\, \tp_2
$$
where $\theta$ is a M\" oebius map such that $\theta(0) = -1/2$ and $\theta(\infty)=1/2$.  Then, we can write $E_a = \cC(x,y,z)$ 
where $x,y,z$ satisfy the relations
$$
z^2 = -\frac{x^2}{(x^2 + 1)^2} \quad\text{ and }\quad y = z^2.
$$
Obviously,  $\cC(x,z)$ is not a strict algebraic extension of $\cC(x)$.  
\end{Example}
From now on, the elements $y = x_n$ and $z = \theta_0(x_0)$ (where $\theta_0$ is the M\" oebius part of $a$) will be called, respectively, the {\em tail }and the
{\em head }elements of the Cohen field of $a$.
\begin{Corollary}\label{corollary-ofcohen}
Assume that $a$ is a non-identity algebraic normal form of affine type, with Cohen field $E$.  Then, the head and tail elements $z,y \in E$ cannot satisfy the relation 
$$
\frac{y}{z} = 1.
$$ 
\end{Corollary}
\begin{proof}
We consider the following three possible cases:
\begin{itemize}
\item[(i)] $a$ is not a rational path.
\item[(ii)] $a$ is a rational path of length $n \ge 1$. 
\item[(iii)] $a = \theta$ is a M\" oebius map.
\end{itemize} 
In the case (i), the result is an immediate consequence of Theorem~\ref{theorem-ofcohen}.  In case (ii), we write the the decomposition of $a$ as
$$
a = \theta_0\, \tp_{u_1}\, \theta_1\, \cdots\, p_{u_n}\, \theta_n
$$ 
for some integers $u_1,\ldots,u_n \in \Omega \cap \cZ$.  It follows that $E$ is the function field of an algebraic curve $C \subset (\cPC)^{n+1}$ which defined by the vanishing 
of  the ideal $\maxn \subset \cC[X_0,\ldots,X_n]$ generated by the equations 
$$
x_{n-1} = \theta_n(x_n)^{u_n}\, ,\, \ldots\, ,\, x_0 = \theta_1(x_1)^{u_1}.
$$
(after appropriately eliminating the denominators in the expression of the M\" oebius maps).  In particular,  the rational map $y = x_n$ defines a degree 1 unbranched covering $C \to \cPC$ while the
rational map $z = \theta(x_0)$ defines a branched covering $C \to \cPC$ of degree $u = |u_1 \cdots u_n| \ge 2$.  Therefore, $y/z$ is a non-constant rational map on $C$.

Finally, in the case (iii), we have the identity $x_0 = x_n = y$.  Therefore we can write 
$$\frac{z}{y} = \frac{\theta(y)}{y}$$
and the right hand side is non-constant because $\theta \ne \id$.  This concludes the proof.
\end{proof}
\begin{Remark}\label{remark-specialusefulcase}
The Example \ref{example-dihedralT2} shows that Corollary~\ref{corollary-ofcohen} is false if we drop the condition that $a$ is of affine type. 
\end{Remark}
\subsection{The field associated to $g \in \NF$}\label{subsect-fieldofg}
We will now generalize the construction of the previous section to arbitrary normal forms.  Assume that $g \in \NF$ has an
algebro-transcendental decomposition
$$
g = a_0\, \gamma_1\, a_1\, \cdots\, \gamma_m\, a_m, \quad m \ge 0
$$
Then, we inductively construct a chain of subfields in $\Merom$,
$$
L_m \subset F_m \subset \cdots \subset L_0 \subset F_0 = F.
$$
as follows.  Firstly, we define $L_m = \cC(y_m)$ to be the field defined over $\cC$ by the identity germ $y_m = \vphi(\id)$.  Then, we put: 
\begin{itemize}
\item[(1)] $F_i = L_i(z_i)$, where $z_i \in \Merom$ is the germ defined 
$$
z_i = \vphi \big( a_i\, \gamma_{i+1}\, a_{i+1}\,\cdots\, \gamma_m\, a_m \big), \quad i = 0,\ldots,m
$$
\item[(2)] $L_i = F_{i+1}(y_i)$, where $y_i \in \Merom$ is the germ defined 
$$
y_i = \vphi \big(\gamma_{i+1}\, a_{i+1}\,\cdots\, \gamma_m\, a_m \big), \quad i = 0,\ldots,m-1
$$
\end{itemize}
We will say that $F = \cC(y_0,z_0,\ldots,y_m,z_m)$ is the {\em field associated to $g$}.

Notice that each field extension $F_{i+1} \subset L_i$ (described in item (2)) is obtained by adjoining to $F_{i+1}$ a (germ of) nonzero solution $y$ to one of the 
following differential equations
$$
\frac{\partial(y)}{y} =  \partial(z), \quad \partial(y) = \frac{\partial(z)}{z} \quad \text{Êor }\quad \frac{\partial(y)}{y} = \alpha \frac{\partial(z)}{z}
$$ 
where we have written write $z = z_{i+1}$.  These equations correspond respectively to the case where $\gamma_{i+1}$ is given by $\te$, $\tl$ or $\tp_\alpha$, for 
some $\alpha \in \Omega \setminus \cQ$.

On the other hand, each extension $L_i \subset F_i$ (described in item (1)) is algebraic, obtained by adjoining a germ of solution $z_i$ of 
a polynomial equation {\em with coefficients in $\cC(y_i)$}.  

More precisely, if we consider the Cohen field $E$ associated to the maximal algebraic path $a_i$ (as defined in subsection~\ref{subsect-cohenfielddef}), and write $E = \cC(x_0,\ldots,x_n)$, then we can embed $\cC(y_i,z_i)$ into $E$ by setting 
$$
y = x_n, \quad \text{Êand }\quad z = \theta_0(x_0).
$$
In particular, the first part of Theorem~\ref{theorem-ofcohen} can be reformulated as follows: If $g = a$ is an algebraic path of affine type then the fields
$F$ and $E$ coincide.
\begin{Remark}
It follows from the above discussion that the differential field $(F,\partial)$ associated to a normal form is a Liouvillian extension of $(\cC(x),\partial)$ (see \cite{MR1960772}, section 1.5). Furthermore, one always have
$$
\mathrm{trdeg}_\cC\, F \le m+1
$$
where $m = \het(g)$.
\end{Remark}
The second part of the Main Theorem will be a direct consequence of Corollary~\ref{corollary-ofcohen} and the following result.
\begin{Theorem}\label{theorem-allsimple} 
Let $g \in \NF_\tame$ be a tame normal form of height $m \ge 1$ and associated field $F = \cC(y_0,z_0,\ldots,y_m,z_m)$.  Then
$$
\mathrm{trdeg}_\cC\, F = m+1.
$$
Moreover, $\{y_0,\ldots,y_m\}$ forms a transcendence basis for $F/\cC$. 
\end{Theorem}
\subsection{Three Lemmas on twisted equations}\label{subsect-twistedequations}
The proof of Theorem~\ref{theorem-allsimple} uses induction on the height of a normal form and  is essentially based on the three Lemmas stated below,
which treat special types of differential equation in $(F,\partial)$.  
For future reference, the equations (\ref{firsttwisteq}), (\ref{secondtwisteq}) and (\ref{thirdtwisteq}) below will be called, respectively,  the {\em first}, {\em second} and {\em third}  twisted equations.

To fix the notation, we consider a  tame normal form $g \in \NF_\tame$ of height $m \ge 0$ with algebro-transcendental decomposition
$$
g = a_0\, \gamma_1\, a_1\, \cdots\, \gamma_m\, a_m, \quad m \ge 0
$$
and associated differential field $F = \cC(y_0,z_0,\ldots,y_m,z_m)$  (equipped with the usual derivation $\partial = d/dz$ induced from $\Merom$).

In the following statements, we use the following definition.  Given a germ $\gamma \in \{\te, \tl, \tp_\alpha \colon \alpha \in \Omega \setminus \cQ\}$, 
the {\em $\gamma$-augmention of $g$ } (or, shortly, the $\gamma$-augmented path) is the path 
$$ 
g_\aug = \gamma \, a_0\, \gamma_1\, a_1\, \cdots\, \gamma_n\, a_n
$$
obtained by adjoining $\gamma$ to the left of $g$ (with the obvious compatibility condition that the source of the germ $\gamma$ coincides with the target of the germ $\theta_0$).   

Note that $g_\aug$ is not necessarily in normal form.  We will say that $g_\aug$ is a {\em nice augmentation }of $g$ if $g_\aug$ lies in $\NF_\tame$ and moreover the right hand side of the above displayed equation is precisely the algebro-transcendental decomposition of $g_\aug$.
 \begin{Lemma}\label{lemma-twisted1}
Assume $\{y_0,\ldots,y_m\}$ is a transcendence basis for $F/\cC$ and suppose that 
the {\em $\te$-augmented path}
$$
g_\aug = \te\, \theta_0\, \gamma_1\, \theta_1\, \cdots\, \gamma_n\, \theta_n
$$
is a nice augmentation of $g$.  Let $f \in F$ be a nonzero solution of the equation
\begin{equation}\label{firsttwisteq}
\Big( \partial - \mu \partial(z_0) \Big) f = 0
\end{equation}
for some $\mu \in \cC$.  Then, necessarily $\mu = 0$ and $f \in \cC$. 
\end{Lemma}
\begin{Lemma}\label{lemma-twisted2}
Assume $\{y_0,\ldots,y_m\}$ is a transcendence basis for $F/\cC$ and suppose that 
the {\em $\tp$-augmented path}
$$
g_\aug = \tp_\alpha\, \theta_0\, \gamma_1\, \theta_1\, \cdots\, \gamma_n\, \theta_n, \quad\text{Êfor some $\alpha \in \Omega \setminus \{1\}$}
$$
is a nice augmentation of $g$.
Let $f \in F$ be a nonzero 
solution of the equation
\begin{equation}\label{secondtwisteq}
\Big( \partial - \mu \frac{\partial(z_0)}{z_0} \Big) f = 0
\end{equation}
for some $\mu \in \cC \setminus \cQ^*$.   Then, $\mu = 0$ and  
$f \in \cC$.
\end{Lemma}
\begin{Lemma}\label{lemma-twisted3}
Assume $\{y_0,\ldots,y_m\}$ is a transcendence basis for $F/\cC$ and suppose that 
the
{\em $\tl$-augmented path}
$$
g_\aug = \tl\, \theta_0\, \gamma_1\, \theta_1\, \cdots\, \gamma_n\, \theta_n
$$
is a nice augmentation of $g$.  
Let $f \in F$ be a nonzero solution of the equation 
\begin{equation}\label{thirdtwisteq}
\partial f = c\, \frac{\partial\big(z_0\big)}{z_0}
\end{equation}
for some $c \in \cC$.  Then, necessarily $c = 0$ and $f \in \cC$. 
\end{Lemma}
We postpone the proofs of these Lemmas to the subsection~\ref{subsect-prooftwisted}.  
\subsection{The proof of Theorem~\ref{theorem-allsimple}}
Let us see how the previous results imply the Theorem~\ref{theorem-allsimple}.
\begin{proof}[Proof of Theorem~\ref{theorem-allsimple}]
As we said above, the proof is by induction in $m = \het(g)$.  Therefore, we start with the case $m = 0$.  By definition, $g = a$ is an algebraic path of affine type, and 
the associated field $K = \cC(y,z)$ has transcendency degree one over $\cC$.

Now, given $m \ge 0$, we assume by induction that all tame normal forms of height $\le m$ satisfy the conclusions of the Theorem.
Let $h \in \NF_\tame$ be a normal form of height
$m+1$.  Then, we can write the decomposition
$$
h = a_0\, \gamma_0\, g
$$
for some algebraic path $a_0 \in \PSL$, a germ $\gamma_0 \in \{\te,\tl,\tp_\alpha : \alpha \in \Omega\setminus \cQ\}$.  Further,  
$$g = a_1\, \gamma_2\, \theta_2\, \cdots\, \gamma_{n+1}\, \theta_{n+1}$$ 
is a tame normal form of height $m$ and the augmented path $g_\aug = \gamma_0 g$ is a nice augmentation of $g$ (see definition at subsection~\ref{subsect-twistedequations}).  

Let us denote by $F$ and $F^\prime$ the differential fields associated to $g$ and $h$, and write
$$
F = \cC(y_1,z_1,\ldots,y_{m+1},z_{m+1}), \quad F^\prime = \cC(y_0,z_0,\ldots,y_{m+1},z_{m+1}). 
$$
Recall that, by construction, $z_0$ is algebraic over $\cC(y_0)$.  Therefore,  to prove the induction step, it suffices to show that the element 
$y_0$ is transcendental over
$F$.   To simplify the notation, from now on we will write $y = y_0$ and $z = z_1$.

Let us assume for a contradiction, that $y$ is algebraic over $F$.
We choose a minimal polynomial $f \in F[X]$ for $y$, say  
$$
f = f_0 + \cdots +  f_{d-1} X^{d-1} + X^d, \quad f_k \in F, 
$$
where we can assume that $d \ge 1$ and that $f_0$ is nonzero.  We  discuss separately the cases where 
$\gamma_0 = \te$ (extension of exp type), $\gamma_0 = \tl$ (extension of log type) and $\gamma_0 = \tp_\alpha$ (extension of power type).  

Suppose that the extension $F \subset F^\prime$ is of exp type.  Recalling that $y$ satisfies the equation $\partial(y) = y \partial(z)$,
it follows from the equation $\partial(f(y)) = 0$ that the polynomial 
$$
p = \partial(f_0) + \cdots + \Big(\partial(f_{d-1}) + (d-1)\partial(z) f_{d-1}\Big) X^{d-1} +  d \partial(z) X^d
$$
vanishes on $y$.  As a consequence,  the polynomial of degree at most $d-1$ given by
$q = p -  d \partial(z) f$ 
also vanishes $y$.  By the minimality of $f$, this polynomial must vanish identically.  This is equivalent to the collection of equations
$$
\Big( \partial - (k-d) \partial(z) \Big) f_k = 0, \quad k = 0,\ldots,d-1.
$$ 
We claim that this implies $d = 0$, which contradicts the definition of $f$.  

Indeed, assume for a contradiction that $d \ge 1$. Then, we are precisely in the hypothesis of  Lemma \ref{lemma-twisted1}, i.e.~each $f_k$ satisfies the first twisted equation with $\mu =  (k-d)$.   In particular, for $k = 0$, one has 
$$\Big( \partial - d \partial(\theta(x))\Big)f_0 = 0$$ 
and the Lemma implies that $f_0 = 0$, which is absurd.

Assume now that the extension $F \subset F^\prime$ is of power type.  Then, it follows from the relation 
$\partial(y) = \alpha\, y\, \partial\big(z\big)/z$ that the polynomial
$$
p = \partial(f_0) + \cdots + \Big(\partial(f_{d-1}) + \alpha (d-1)\frac{\partial(z)}{z} f_{d-1}\Big) X^{d-1} +  \alpha d \frac{\partial(z)}{z} X^d
$$
also vanishes on $y$.  Hence, by the minimality of $d$, the polynomial
$$
q = p -  \alpha d \frac{\partial(z)}{z} f$$
must vanish identically.  This corresponds to the collection of equations
$$
\Big( \partial - \alpha (d-k) \frac{\partial{(z)}}{z}\Big) f_k = 0, \quad k = 0,\ldots,d-1.
$$
We claim that this set of equations has no solution if $d \ge 1$.  Indeed, in this case, each one of the above equation corresponds to a twisted equation as described in 
Lemma~\ref{lemma-twisted2} with $\mu = \alpha (d-k) \in \cC \setminus \cQ$.  In particular, for $k = 0$, we conclude from that Lemma that $f_0 = 0$, which is absurd. 

Finally, in case where the extension $F \subset F^\prime$ is of log type, we have 
$$
\partial(y) = \frac{\partial(z)}{z}.
$$
Hence, the polynomial of degree at most $d-1$ given by
$$
q = \Big( \partial(f_0) + \frac{\partial(z)}{z} f_1 \Big) + \cdots + \Big(\partial(f_{d-1}) + d \frac{\partial(z)}{z} \Big) X^{d-1} 
$$
also vanishes on $y$.  By the minimality of $d$, this polynomial vanishes identically and this corresponds to say that the collection of equations 
$$
\partial f_{k} =  \frac{\partial(z)}{z}(k+1) f_{k+1}, \quad \text{ for }k = 0,\ldots,d-1
$$
hold, where we put $f_d = 1$.  Taking $k = d -1$, the rightmost equation is exactly the twisted equation from 
Lemma~\ref{lemma-twisted3}.   From the Lemma, it follows that $d=0$, which contradicts our assumption. This concludes the proof of the Theorem. 
\end{proof}
\begin{Remark}\label{remark-casep=0}
As some readers may have noticed, the above computations are very similar to the classical computations of the Picard-Vessiot extension $K/k$ for the elementary  linear differential equations 
$$\partial(y) = a, \quad \text{or} \quad \partial(y) = a y, \quad \text{ with }a \in k$$ 
over a given differential field $(k,\partial)$ (see e.g.~\cite{MR1960772}, examples 1.18 and 1.19).  In this simple setting, the computation of the differential Galois group of the extension $K/k$ reduces to studying when these equations have no solution in the base field.
\end{Remark}
\subsection{Resonance relations and Ax Theorem}
In this subsection, we  recall some results about differential field extensions and differentials forms, following closely Wilkie's notes \cite{wilkie}. They are key ingredients in proof of the celebrated Ax's Theorem (cf. \cite{MR0277482}).
We observe that the results described here are completely independent of the Theorem \ref{theorem-allsimple} and Lemmas \ref{lemma-twisted1}, \ref{lemma-twisted2} and \ref{lemma-twisted3}.  

In this subsection, $k \subseteq K$ will denote arbitrary fields of characteristic zero, and $k$ will be assumed to be algebraically closed. 
We will say that elements $x_1,\ldots,x_n \in K$ satisfy a {\em power resonance relation (over $k$) }if there exist integers $r_1,\ldots,r_n$, not all zero, 
such that
$$
\prod_{i=1}^n x_i^{r_i}  \in k.
$$
Similarly, we will say that $y_1,\ldots,y_n \in K$ satisfy a {\em linear resonance relation (over $k$) }if there exist integers $r_1,\ldots,r_n$, not all zero, 
such that
$$
\sum_{i=1}^n r_i\, y_i \in k.
$$
Since $k$ is supposed algebraically closed, we can assume in both cases that the integers $r_1,\ldots,r_n$ are coprime.

We will denote by 
$\mathrm{Der}_k(K)$ the set of derivations $\delta: K \rightarrow K$ whose constant subfield $\Const(\delta)$ contains $k$.  For each $n \in \cN$, $\Omega^n_k(K)$ denotes the $K$-vector space of alternating, $K$-linear $n$-forms on $\mathrm{Der}_k(K)$ and 
$$d: \Omega^n_k(K) \longrightarrow \Omega^{n+1}_k(K)$$
is the total differential map.  The space $\Omega^1_k(K)$ is the dual of $\mathrm{Der}_k(K)$, and it is usually called the space of K\" ahler differentials of $K$ over $k$.   The space $\Omega^0_k(K)$ is identified to $K$.

We say that a 1-form 
$\omega \in \Omega^1_k(K)$ is closed (resp. exact) if $d\omega = 0$ (resp. 
$\omega = du$ for some $u \in K$).  Finally, given an intermediate field $k \subset K_0 \subset K$,  we say that a form $\omega \in \Omega^1_k(K)$ is {\em defined over $K_0$} if $\omega = \sum {a_i db_i}$, for some $a_i,b_i \in K_0$.
\begin{Lemma}\label{lemma-lindepforms}
Suppose that $K_0$ is a field such that $k \subset K_0 \subset K$ and $\transdeg_k(K_0) = n$ for some $n \ge 1$.  Let $\delta \in \Der_k(K)$ be a derivation such that $\Const(\delta) = k$, and suppose that 
$$\omega_1,\ldots,\omega_n \in \Omega^1_k(K)$$ 
are closed 1-forms defined over $K_0$ satisfying $\omega_i (\delta) = 0$, for $i  =1,\ldots,n$. Then $\omega_1,\ldots,\omega_n$ 
are linearly dependent over $k$.
\end{Lemma}
The above statement and its proof can be found, for instance, in 
Wilkie's notes \cite{wilkie}, Theorem 2.
\begin{Lemma}\label{lemma-monomial}
Suppose that there exists nonzero $x_1,\ldots,x_m \in K$ and elements $e_1,\ldots,e_m \in k$ not all zero such that the differential form
$$
\sum_{i=1}^m e_i \frac{d x_i}{x_i} 
$$
is exact.  Then, $x_1,\ldots,x_m$ satisfy a power resonance relation over $k$.  Moreover, assume that there exists a derivation $\delta \in \Der_k(K)$ and $y_1,\ldots,y_m \in K$ such that
$$\Const(\delta) = k, \quad \text{Êand }\quad \frac{\delta(x_i)}{x_i} = \delta(y_i)$$ 
for each $1 \le i \le m$.  Then $y_1,\ldots,y_m$ satisfy a linear resonance relation over $k$. 
\end{Lemma}
\begin{proof}
The first statement is proved in \cite{wilkie}.  For the second statement, it suffices to remark that if the monomial 
$m = \prod x_i^{r_i}$ belongs to $k$ then the linear form $l = \sum r_i y_i$ satisfies 
$$
\delta(l) = \sum r_i \delta(y_i) = \sum r_i \frac{\delta(x_i)}{x_i} = \frac{\delta(m)}{m} = 0
$$
which implies that $l \in k$.
\end{proof}
\subsection{Resonances and transcendental equations}\label{subsect-resonance-transcendental}
We now consider an arbitrary normal form $g \in \NF$ of height $m \ge 0$, with algebro-transcendental decomposition,
$$
g = a_0 \, \gamma_{1} \, a_1\, \cdots \, \gamma_m \, a_m, \quad m \ge 0
$$
and associated differential field $(F,\partial)$, where we write
$$
F = \cC(y_0,z_0,\ldots,y_m,z_m)
$$ 
as in subsection~\ref{subsect-fieldofg}.  Further, for each $i = 0,\ldots,m$, we denote by $F_i = \cC(y_i,z_i)$ the field associated to the maximal algebraic subpath $a_i$.\begin{Proposition}\label{prop-resonantmonomial}
Assume that a nonzero element $f \in F$ satisfies one of the following three equations
$$
(1)\; \frac{\partial(f)}{f} = \partial(z_0), \quad  (2)\; \partial(f)= \frac{\partial(z_0)}{z_0}\quad\text{or}\quad (3)\; \frac{\partial(f)}{f} = \beta\frac{\partial(z_0)}{z_0}
$$
for some $\beta \in \cC \setminus \cQ$. Then, the elements 
$$y_0,z_0,\ldots,y_m,z_m \in F$$ 
satisfy either a power resonant relation or a linear resonant relation over $\cC$. 
\end{Proposition}
\begin{proof}
Based on the algebraic-transcendental expansion of $g$ written above, 
we consider the following subsets of $\{1,\ldots,m\}$
\begin{align*}
I_e &= \{ j  \colon \gamma_j = \te\}, \quad I_l = \{ j \colon \gamma_j = \tl\},\quad \text{Êand }\\
I_p &= \{ j \colon \gamma_j = \tp_{\alpha_j}, \text{Êfor some }\alpha_j \in \Omega \setminus \cQ\},
\end{align*}
and define a collection of closed 1-forms $\omega_1,\ldots,\omega_m \in \Omega^1_{\cC}(F)$ as follows
$$
\omega_j = 
\begin{cases}
dy_{j-1}/y_{j-1} - dz_j,& \text{Êif }j \in I_e\\
dy_{j-1} - dz_j/z_j,& \text{Êif }j \in I_l\\
dy_{j-1}/y_{j-1} - \alpha_j dz_j/z_j,& \text{Êif }j \in I_p\\
\end{cases}
$$
where, in this last case, $\alpha_j$ denotes the exponent in the power map $\gamma_j = \tp_{\alpha_j}$.  Similarly, we define the closed 1-form $\omega_0$ as
$$
\omega_0 = 
\begin{cases}
df/f - dz_0,& \text{Êif $f$ satisfies (1)}\\
df - dz_0/z_0,& \text{Êif $f$ satisfies (2)}\\
df/f - \beta dz_0/z_0,& \text{Êif $f$ satisfies (3)}\\
\end{cases}
$$
Then, by Lemma~\ref{lemma-lindepforms}, since $\omega_j(\partial) = 0$ for all $0 \le j \le m$, and $\mathrm{trdeg}_\cC(F) \le m+1$, there exist constants $c_0,\ldots,c_m \in \cC$, not all zero, such that the relation
$$c_0 \omega_0 + \cdots + c_m \omega_m = 0$$  
holds.  We consider, first of all, the case where $f$ satisfies equation (2).  Then, by suitably regrouping the terms in the above relation, we obtain a 1-form 
$$
c_0 \frac{dz_0}{z_0} + \sum_{j \in I_e}{c_i\, \frac{dy_{i-1}}{y_{i-1}}}+ \sum_{j \in I_l}{c_j\, \frac{dz_j}{z_j} }+ \sum_{j \in I_p}{c_j\, \Big( \frac{dy_{j-1}}{y_{j-1}} + \alpha_j
\frac{dz_j}{z_j} \Big)}
$$
which is exact.  Therefore, we can apply Lemma \ref{lemma-monomial} to conclude that there exists a monomial
$$
M = z_0^{v_0} \prod_{j \in I_e}{y_{j-1}^{u_{j-1}}}\; \prod_{j \in I_l}{z_j^{v_j}} \; \prod_{j \in I_p}{y_{j-1}^{u_{j-1}}z_j^{v_j}}
$$
with integer exponents not all zero, which belongs to $\cC$.  This proves the result.

We consider now the case (1).  Here, we conclude from the relation  $\sum {c_j \omega_j}=0$ that the $1$-form
$$
c_0 \frac{df}{f} + \sum_{j \in I_e}{c_j\, \frac{dy_{i-1}}{y_{i-1}}}+ \sum_{j \in I_l}{c_j\, \frac{dz_j}{z_j} }+ \sum_{j \in I_p}{c_j\, \Big( \frac{dy_{j-1}}{y_{j-1}} - \alpha_j
\frac{dz_j}{z_j} \Big)}
$$
is exact.  Hence, applying again Lemma \ref{lemma-monomial}, we show that there exists a monomial of the form
$$
M = f^w \prod_{j \in I_e}{y_{j-1}^{u_{j-1}}}\; \prod_{j \in I_l}^m{z_j^{v_j}} \; \prod_{j \in I_p}{y_{j-1}^{u_{j-1}}z_j^{v_j}} \in \cC
$$ 
(for some integers $w,v_j,u_j$, not all zero) which belong to $\cC$.  Notice that if $w = 0$ we are done because this would give the desired relation.  
By the same reason,  we would be done if $c_0 = 0$.

Hence, from now on, we can assume that $w$ is nonzero and that $c_0 = 1$.  By computing the logarithmic derivative $dM/M$, we can write
$$
\frac{df}{f}Ê= \sum_{j \in I_e}^m{p_{j-1} \frac{dy_{j-1}}{y_{j-1}}} + \sum_{j \in I_l}^m{q_j \frac{dz_j}{z_j}} + \sum_{j \in I_p}^m{\Big( p_{j-1} \frac{dy_{j-1}}{y_{j-1}} + q_j
\frac{dz_j}{z_j} \Big)}
$$
with $p_j = - u_j/w$ and $q_j = -v_j/w$ being rational numbers. We can now replace this expression for $df/f$ in the relation $\sum{c_j \omega_j} = 0$ and, again by suitably
regrouping the terms, conclude that the $1$-form
\begin{align*}
\sum_{j \in I_e}{(c_j + p_{j-1})\, \frac{dy_{i-1}}{y_{i-1}}}+ &\sum_{j \in I_l}{(c_j + q_j)\, \frac{dz_j}{z_j} }+ \\
&\sum_{j \in I_p}{(c_j + p_{j-1}) \frac{dy_{j-1}}{y_{j-1}} 
+ (-c_j\alpha_j + q_j)
\frac{dz_j}{z_j}}
\end{align*}
is exact.  If at least one of the coefficients of this 1-form is nonzero then we can apply again Lemma \ref{lemma-monomial} in order to obtain a monomial satisfying the conditions in the statement.
So, let us assume that all these coefficients vanish.  Since $\alpha_i \notin \cQ$ for each $i \in I_p$, we conclude that
$$
c_j = p_{j-1} = q_j = 0, \quad \text{Êfor all }j \in I_p.
$$
In particular, the monomial $M$ has simply the form
$$
M = f^w \prod_{j \in I_e}{y_{j-1}^{u_{j-1}}}\; \prod_{j \in I_l}^m{z_j^{v_j}}.
$$
If we consider the linear form
$$
l = w\, z_0 + \sum_{j \in I_e}{u_{j-1} z_j} + \sum_{j \in I_l}{v_{j} y_{j-1}}
$$
it is easy to see that $\partial(l) = \partial(M)/M = 0$.  Therefore, $l \in \Const(\partial) = \cC$.

It remains to consider the case of equation (3). The treatment is similar to the previous case.  Here, we obtain a 1-form 
$$
c_0 \Big( \frac{df}{f} - \beta \frac{dz_0}{z_0}) + \sum_{j \in I_e}{c_i\, \frac{dy_{i-1}}{y_{i-1}}}+ \sum_{j \in I_l}{c_j\, \frac{dz_j}{z_j} }+ \sum_{j \in I_p}{c_j\, \Big( \frac{dy_{j-1}}{y_{j-1}} + \alpha_j
\frac{dz_j}{z_j} \Big)}
$$
which is exact, and hence there exists a monomial 
$$
M = f^w z_0^{v_0}\prod_{j \in I_e}{y_{j-1}^{u_{j-1}}}\; \prod_{j \in I_l}^m{z_j^{v_j}} \; \prod_{j \in I_p}{y_{j-1}^{u_{j-1}}z_j^{v_j}} \in \cC
$$
(with exponents not all zero).  Assuming that $c_0 = 1$ and that $w$ is nonzero (otherwise we are done), we can apply exactly the same reasoning as above 
to conclude that the 1-form
\begin{align*}
(q_0 - \beta)\frac{dz_0}{z_0} + &\sum_{j \in I_e}{(c_j + p_{j-1})\, \frac{dy_{i-1}}{y_{i-1}}}+ \sum_{j \in I_l}{(c_j + q_j)\, \frac{dz_j}{z_j} }+ \\
&\sum_{j \in I_p}{(c_j + p_{j-1}) \frac{dy_{j-1}}{y_{j-1}} 
+ (-c_j\alpha_j + q_j)
\frac{dz_j}{z_j}}
\end{align*}
is exact (where $q_0 = -v_0/w$).   Since $\beta \notin \cQ$, the coefficient in front of $dz_0/z_0$ cannot vanish. Therefore, we can apply 
Lemma \ref{lemma-monomial} in order to obtain a monomial which satisfies the desired relation.   This concludes the proof.
\end{proof}
For later use, we need to establish a more precise statement about the existence of power/linear relations in the fields
$F_0,\ldots,F_m$ associated to the algebraic subpaths $a_0,\ldots,a_m$.  

Given a normal form $g = a_0\gamma_1a_1
\cdots \gamma_m a_m$ as in the beginning of the subsection and an equation for $f \in F$ as in the statement of the previous Proposition, we define
the {\em augmentation of $g$} as the path  
$$
g_{\aug} =  \gamma_0\, a_0\, \gamma_1\,  \cdots \, \gamma_m\, a_m
$$
which is obtained by concatenating to $g$ the symbol $\gamma_0 = \te$ (resp. $\tl$ or $\tp_\beta$) if $f$ satisfies equation (1) (resp. (2) or (3)).

Further, given an index $0 \le j \le m-1$ and two symbols $\gamma,\gamma' \in \{\te,\tl,\tp\}$, we will say that  the algebraic subpath $a_j$ of $g_{\aug}$ lies in a {\em 
$[\gamma,\gamma']$ segment }if 
$$\gamma_j = \gamma\quad\text{ and }\quad\gamma_{j+1} = \gamma'$$ 
\begin{Corollary}\label{corollary-resonantmonomialstrans}
Assume that $\{y_0,\ldots,y_m\}$ is a transcendence basis for $F/\cC$.  Let $f \in F$ be a non-zero solution of one of the equations (1),(2) or (3) from Proposition~\ref{prop-resonantmonomial}. Then, there exists at least one index $0 \le j \le m-1$ such that
\begin{itemize}
\item[(i)]Either $a_j$ lies in a $[\te,\tl]$ segment and $y_j,z_j$ satisfy a linear resonance relation in $F_j$, 
\item[(ii)] Or $a_j$ lies in a $[\tl,\te]$, $[\tl,\tp]$, $[\tp,\te]$ or $[\tp,\tp]$ segment and $y_j,z_j$ satisfy a power resonance relation in $F_j$.
\end{itemize}
In particular, if $m = 0$ then there is no nonzero element $f \in F$ satisfying (1),(2) or (3).
\end{Corollary}
\begin{proof}
The hypothesis imply that $\{dy_0,\ldots,dy_m\}$ is a 
basis of $\Omega^1_\cC(F)$ and that the $F$-subspaces generated by
$\Omega^1_\cC(F_0),\ldots,\Omega^1_\cC(F_m)$ are $F$-linearly independent.  Moreover, 
since each $z_j$ is algebraic over $y_j$, the 1-form $dz_j$ lies in the one-dimensional $F$-subspace generated by $dy_j$.

From the Proposition~\ref{prop-resonantmonomial}, we conclude that if a nonzero element $f \in F$ satisfies (1), (2) or (3) then either there exists a monomial 
$M= \prod_{j=0}^m{z_j^{v_j}y_j^{u_j}}$ or a linear form $l = \sum_{j=0}^m {v_j z_j + u_j y_j}$ 
(with integers $u_j,v_j$ not all zero) which belong to $\cC$.  Taking the logarithmic derivative $dM/M$ in the former case or the derivative $dl$ in the later case, we obtain 
$$
\sum_{j=0}^m \frac{v_j}{z_j} dz_j + \frac{u_j}{y_j} dy_j = 0, \quad\text{Êor }\quad \sum_{j=0}^m v_j dz_j + u_j dy_j = 0 
$$
respectively.  Therefore, by the linear independency of $dy_0,\ldots,dy_n$, either $\frac{v_j}{z_j} dz_j + \frac{u_j}{y_j} dy_j = 0$ or $v_j dz_j + u_j dy_j = 0$ for all $0 \le j \le m$.   In the former case, we conclude that $d(z^{v_j} y_j^{u_j}) = 0$, while in the latter case $d(v_j z_j + u_j y_j) = 0$.

Now, to conclude the proof, it suffices to consider more carefully the expressions of $m$ and $l$ obtained in the proof of the previous Proposition.  
For instance, we consider the case where 
$$
m = z_0^{v_0} \prod_{j \in I_e}{y_{j-1}^{u_{j-1}}}\; \prod_{j \in I_l}{z_j^{v_j}} \; \prod_{j \in I_p}{y_{j-1}^{u_{j-1}}z_j^{v_j}} \in \cC
$$
which, by the above argument, implies a collection of power resonance relations of the form
$d(z_j^{v_j} y_j^{u_j}) = 0$, for $j = 0,\ldots,n$. 

Notice that no relation of type $d(z_j^{v_j}) = 0$ (i.e. with $u_j = 0$) or $d(y_j^{u_j})$ (i.e. with $v_j = 0$) can appear, since this would imply that $z_j$ or $y_j$ belong to $\cC$, contradicting the fact that both $y_j$ and $z_j$ are germs of invertible maps.  
Thus, there necessarily exists a monomial relation of the form $z_j^{v_j} y_j^{u_j} \in \cC$ with exponents $u_j,v_j$ both nonzero.  But looking to the above expression for $m$, we conclude that this can only happen in the index $j$ is such that
$j \in I_l \cup I_p$ and $j+1 \in I_p \cup I_e$.  This  is equivalent to say that $a_j$ lies in a $[\tl,\te]$, $[\tl,\tp]$, $[\tp,\te]$ or $[\tp,\tp]$ segment. 

The other cases can be treated in an analogous way.
\end{proof}

\subsection{Proofs of Lemmas on twisted equations}\label{subsect-prooftwisted}
We now proceed to the proof of Lemmas \ref{lemma-twisted1}, \ref{lemma-twisted2} and \ref{lemma-twisted3}.  We keep the notation introduced in subsection~\ref{subsect-twistedequations}.
\begin{proof}[Proof of Lemma \ref{lemma-twisted1}]
Let us assume that $\mu \ne 0$.  By contradiction, we assume that there exists a nonzero $f \in F$ such that
$$
\frac{\partial(f)}{f} = \partial(z_0)
$$
Writing the algebro-transcendental decomposition of the $\te$-augmented path $g_\aug$ as
$$
g_\aug = \te\, a_0 \, \gamma_1\, a_1\, \cdots \, \gamma_m\, a_m, \quad m \ge 0
$$
we let $z_j,y_j$ denote the head and tail elements of the Cohen differential field $E_j$ associated to the algebraic path $a_j$, for $j = 0,\ldots,m$.

Defining $\gamma_0 = \te$, we can now apply Corollary \ref{corollary-resonantmonomialstrans} to conclude that there exists at least one index $0 \le j \le m-1$ such that
\begin{itemize}
\item[(1)] Either $\gamma_j = \te$, $\gamma_{j+1} = \tl$ and $y_j,z_j$ satisfy a linear resonance relation, 
\item[(2)] Or $\gamma_j \in \{\tl,\tp_\alpha : \alpha \in \Omega \setminus \cQ\}$, $\gamma_{j+1} \in \{\te,\tp_\alpha : \alpha \in \Omega \setminus \cQ \}$ and $y_j,z_j$ satisfy a power resonance relation.
\end{itemize}
If $m = 0$ we get our desired contradiction.  If $m \ge 1$, we will deduce the contradiction using Corollary~\ref{corollary-ofcohen}.

For this, we treat cases (1) and (2) separately.  To simplify the notation, we define
$$a = a_j, \quad y = y_j\quad \text{Êand } \quad z = z_j.$$  
and write the expansion of the algebraic path $a$ (of affine type) as
$$
a = \theta_0 \,  p_1 \,  \theta_1 \,  \cdots \,  p_n \,  \theta_n, \quad n \ge 0.
$$
In the case (1), $y$ and $z$ satisfy a relation of the form $v z  + u y = c$, for some  $u,v \in \cZ^*$ and $c \in \cC$.  
We consider then the {\em modified algebraic path}
$$
a^* = \text{ normal form reduction of }\ts_{-v/u}\,  \tt_{-c/v}\, a.
$$
Explicitly, for $a$ given as above, we can write 
$$
a^* = \theta^*_0 \,  p_1 \,  \theta_1 \,  \cdots \,  p_n \,  \theta_n
$$
where the M\" oebius part of $a^*$ is given by $\theta^*_0 = \ts_{-v/u}\,  \tt_{-c/v} \theta_0$.  In particular, the assumption that $a$ is an algebraic path of affine type implies that the same property holds for $a^*$.   

Now, by the definition of $a^*$, the head and tail elements $z^*$ and $y^*$ of the Cohen field $E^*$ associated to $a^*$ should satisfy the relation 
$$\frac{y^*}{z^*} = 1$$
Hence, we will obtain the desired contradiction to Corollary~\ref{corollary-ofcohen} once we show that
$$a^* \ne \id.$$  
To prove that this always holds, observe that $a^* = \id$ if and only if $a = (\ts_{-v/u}\,  \tt_{-c/v})^{-1} = \tt_{c/v}\ts_{-u/v}$.  Since $a$ is a maximal algebraic subpath (lying in a $[\te,\tl]$ segment) of the augmented path $g_\aug$, this would contradict the hypothesis that $g_\aug$ is a nice augmentation of $g$, 
as stated in Subsection~\ref{subsect-twistedequations} .  

Indeed, if either $c \ne 0$ or $-u/v \notin \Omega \setminus \{1\}$ then the subpath $\te a \tl$ is certainly not in normal form.  On the other hand, if $c = 0$ and $-u/v \in \Omega \setminus \{1\}$ then, according to our definition of algebro-transcendental decomposition the corresponding subpath $\te \ts_{-u/v} \tl$ should instead be considered as a rational power map 
$\tp_{-u/v}$.  This concludes the proof of (1).

Consider now the case (2). We write the corresponding power resonance relation as $y^u z^v = c$, for some $u,v \in \cZ^*$ and $c \in \cC^*$.
Since the algebraic path $a$ lies in a $[\gamma,\eta]$-segment (with $\gamma \in \{\tl,\tp\}$ and $\eta \in \{\te,\tp\}$) and $g_\aug$ is a nice augmentation of $g$, it follows from the definition of $\NF$ that its M\" oebius part $\theta_0$ necessarily lies in  $\trans_1 \setminus \{\id\}$.  

We introduce now the modified algebraic path
$$
a^* = \ts_{c^{1/u}}\tp_{v/u}\, a.
$$
which is also a non-identity normal form by the discussion of the above paragraph. Similarly to the previous case, the head and tail elements $z^*,y^*$ of the Cohen field $E^*$ associated to $a^*$ satisfy the relation 
$\frac{y^*}{z^*} = 1$
(up to a convenient choice of the branch of $c^{1/u}$).  Furthermore, $a^*$ is an algebraic normal form of affine type and the above identity contradicts Corollary~\ref{corollary-ofcohen} when applied to $a^*$.  This concludes the proof of the Lemma. 
\end{proof}

\begin{proof}[Proofs of Lemmas \ref{lemma-twisted2} and \ref{lemma-twisted3}]  We follow exactly the same strategy of the previous proof.  

Namely,  we assume for a contradiction that $c \ne 0$ and  
that there exists a nonzero element $f \in F$ satisfying one of the following two equations
$$
\frac{\partial(f)}{f} = \mu \frac{\partial(z_0)}{z_0}\quad \text{ or }\quad \partial(f) = \frac{\partial(z_0)}{z_0} \; 
$$
where $\mu \in \cC \setminus \cQ$. Considering the algebraic transcendental decomposition of the augmented path $g_\aug$, the same alternatives 
(1) and (2) listed in the previous proof appear. By repeating the same reasoning, we obtain a contradiction. 
\end{proof}
\section{Some consequences}
We proceed to prove the other results stated in the Introduction. 
\begin{proof}[Proof of Theorems~\ref{theorem-amalgamatedpoweraffine}, \ref{theorem-NFinWitt} and \ref{theorem-NFinmonoid}]
We will only prove Theorem~\ref{theorem-amalgamatedpoweraffine}, since the other two results are immediate consequences.

First of all, we remark that $\Gr_{\Aff,\Pow_R}$ is a subgroupoid of $\Gr_{\PSL,\Exp}$. 
Therefore, if we consider the free product groupoid $\Free = \Gr_{\PSL}\, *\, \PGr_{\Exp}$ and the groupoid morphism 
$$\vphi: \Free \to \Gr_{\PSL,\Exp}$$ 
defined in subsection \ref{subsect-listrel}, then each germ lying in the $\Gr_{\Aff,\Pow_R}$ is the image of a (not necessarily unique) path in $\Free$.  Further, we can assume that such path of the form
$$
g = \theta_0 \, \tp_{r_1} \, \theta_1\, \cdots\, \tp_{r_n} \, \theta_n, \quad n \ge 1
$$
where each $\tp_{r_i}$ is a power map with exponent  $r_i \in R$ and each $\theta_i$ is an affine map.  Possibly making some simplifications,  
we can further assume that $\theta_1,\ldots,\theta_{n} \ne \id$ and that $r_1,\ldots,r_n \ne 1$. 

As a consequence, $g$ is a {\em product normal form}, i.e.~an element of the subset $\PNF \subset \Free$ given by Definition~\ref{def-productnf}.  Applying 
the reduction system $(\PNF,\rightarrow)$ defined in subsection~\ref{subsect-confluentnf}, we can make the reduction 
$$
g \rightstar g^\prime
$$
where $g^\prime$ has the same form as $g$, but with the additional property that {\em each affine map $\theta_1,\ldots,\theta_n$ is a translation}.
The subset of paths in $\Free$ satisfying these properties will be called {\em normal forms of power-translation type}, and noted $\NFPT$.

Notice that a path in $\NFPT$ is not necessarily an element of $\NF$ (see definition~\ref{def-nf}), because the exponents $r_1,\ldots,r_n$ of the power maps do not necessarily lie in the region $\Omega$ described in Remark~\ref{remark-defNF}.

However, the reduction from $\NFPT$ to $\NF$ can be easily obtained.  Indeed, assuming that $g \in \NFPT$ is written as above, its normal form reduction $g \rightstar h$ gives the path
$$h = \theta_0 \ti^{\eps_1} \, \tp_{s_1} \, \theta_1\ti^{\eps_2}\, \tp_{s_2} \cdots\, \ti^{\eps_n}\tp_n \, \theta_n$$
where we define each pair $(s_i,\eps_i) \in \Omega \times \{0,1\}$ as follows: 
$$
(s_i,\eps_i) = 
\begin{cases}
 (r_i,0)  & \text{Êif $r_i \in \Omega$}\\
 (-r_i,1) & \text{ if $-r_i \in \Omega$}.  
\end{cases}
$$
We remark the following two facts:
\begin{itemize}
\item[(i)] The normal form $h$ lies in $\NF_\tame$.
\item[(ii)] If $g,g^\prime \in \NFPT$ reduce to a same normal form $h \in \NF_\tame$ then necessarily $g = g^\prime$.
\end{itemize}
Indeed, the assumption $R \cap \cQ_{<0} = \emptyset$ implies that a subpath of the form $\theta_i\ti$ appears in the expansion of $h$ if and only if the 
power map $\tp_{r_{i+1}}$ has an exponent in $R \setminus \cQ$.  Therefore, the algebro-transcendental decomposition of $h$ can only contain maximal algebraic subpaths of affine type. This proves (i).

The proof of (ii) is immediate, since the original powers $r_1,\ldots,r_n \in R$ can be read out from the expression of the normal form.

Based on these remarks, the result is now an immediate consequence of the second part of the Main Theorem.
\end{proof}
\begin{proof}[Proof of Theorems~\ref{Theorem-realaffineexp} and \ref{theorem-lemme1}]
We will only give the details of the proof of Theorem~\ref{Theorem-realaffineexp}, since Theorem~\ref{theorem-lemme1} is an immediate consequence of this result.

Keeping the notation of subsection~\ref{subsect-HNN-ext}, we want to prove that the homomorphism 
$$
\phi: \Aff^+_{*\theta} \rightarrow \gr_{\Aff^+,\Exp}. 
$$
is injective.  

Using Britton's normal form (see e.g. \cite{MR1812024}, IV.2), and the right transversals to $T_0,T_1$ to $H_0,H_1$ defined in Remark~\ref{remark-defNF}, it follows that 
we can (setwise) identify $\Aff^+_{*\theta}$ to a set
$\BNF$ (so-called {\em Britton normal forms}) contained in the free product $\Aff^+ \,*\, \langle \stable^i : i \in \cZ \rangle$ 
(where $\stable$ denotes the stable letter of the HNN-extension).  By definition, each $f \in \BNF$ can be uniquely written as
$$
f = \theta_0\, \gamma_1 \, \theta_1\, \cdots\, \gamma_n\, \theta_n, \quad n \ge 0
$$
where $\theta_0,\ldots,\theta_{n}$ are affine maps, $\gamma_i \in \{\stable,\stable^{-1}\}$, and
\begin{itemize}
\item[(i)] If $\gamma_i = \stable$ then $\theta_i \in S^+$
\item[(ii)] If $\gamma_i = \stable^{-1}$ then $\theta_i \in  T$.
\item[(iii)] There are no subwords of the form $\stable \id \stable^{-1}$ or $\stable^{-1}\id \stable$.
\end{itemize}
The set $\BNF$ has a natural group structure which is inherited from the group structure of $\Aff^+_{*\theta}$.

Using the above expansion for $f \in \BNF$, we define {\em maximal interval of existence }$I_f \in (\cR,+\infty)$ of $f$ as the largest open neighborhood of $\infty$ 
(of the form $]A_f,+\infty[$ for some $A_f \in \cR$) such that each one of the $n+1$ {\em truncations} of the above normal form, namely
$$
f^{[i]} = \theta_i\, \gamma_{i+1}\, \theta_{i+1}\cdots\, \gamma_n\, \theta_n\, \quad \text{Êfor }i = 0,\ldots,n
$$
maps under $\phi$ to a germ $\phi(f^{[i]}) \in \gr_{\Aff^+,\Exp}$ which extends analytically to an invertible function defined on the interval $I_f$. To simplify the notation, we denote also by
$$
\phi(f) \colon I_f \to \cR
$$
the corresponding (uniquely determined) analytic function.  

Now, we consider a mapping $\rho_1 :  \BNF \to \NF_\tame$ which sends an element $f \in \BNF$ to a tame normal form 
$\rho_1(f) \in \NF_\tame$.  If $f$ is  written as above, this mapping is defined as follows: 
\begin{itemize}
\item[(a)] Each symbol $\theta_i$ is replaced by a corresponding germ of affine map;
\item[(b)] Each symbol $\stable$ (resp.~$\stable^{-1}$) is replaced by a germ of exponential (resp.~principal branch of logarithm) map;
\item[(c)]  The source point of the rightmost affine germ $\theta_n$ is chosen to be $A_f+1$ (or $0$ if $A_f = -\infty$).
\end{itemize}
Notice that condition (c) uniquely determines the choice of all germs given in (a) and (b) due to the necessarily source/target compatibility conditions.  Consequently, the mapping is well-defined by these conditions and, moreover, injective.

Similarly, we consider the mapping
$$\rho_2 :  \BNF \to \Gr_{\PSL,\Exp}$$
defined as follows:  given $f \in \BNF$, we consider the analytic function $\phi(f) : I_f \to \cR$ and let
$\rho_2(f) \in \Gr_{\PSL,\Exp}$ be the germ of  $\phi(f)$ at the point $A_f+1$ (or at $0$ if $A_f = -\infty$). 

By construction, if $\vphi : \NF \to \Gr_{\PSL,\Exp}$ denotes the mapping defined at the Main Theorem, the following diagram 
$$
\begin{tikzpicture}
  \matrix (m) [matrix of math nodes,row sep=3em,column sep=5em,minimum width=2em,text height=1.5ex, text
depth=0.25ex]
  {
     \BNF & \NF_\tame \\
              &   \Gr_{\PSL,\Exp}\\};
  \path[-stealth]
    (m-1-1) edge node [above] {\small $\rho_1$} (m-1-2)
            edge  node [above]  {\small $\rho_2$} (m-2-2)
    (m-1-2) edge node [right] {\small $\vphi$} (m-2-2);
\end{tikzpicture}
$$
is commutative.

Now, we reason by contradiction assuming that there exists a non-identity Britton normal form $f \in \BNF$ lying in the kernel of $\phi$.  
Then, it follows that $\phi(f) : I_f \to \cR$ is the identity map and, consequently, that $\rho_2(f)$ is the identity germ.  
On the other hand, $\rho_1(f)$ is a non-identity tame normal form and it follows from the Main Theorem that $\vphi\circ\rho_1(f)$ cannot be the identity germ.  This is
a contradiction.
\end{proof}
\bibliographystyle{plain}
\bibliography{mybibliography}
\Addresses
\end{document}

%% file: exponentialfol.pdf_t
\begin{picture}(0,0)%
\includegraphics{exponentialfol.pdf}%
\end{picture}%
\setlength{\unitlength}{3947sp}%
\begingroup\makeatletter\ifx\SetFigFont\undefined%
\gdef\SetFigFont#1#2#3#4#5{%
  \reset@font\fontsize{#1}{#2pt}%
  \fontfamily{#3}\fontseries{#4}\fontshape{#5}%
  \selectfont}%
\fi\endgroup%
\begin{picture}(5451,2452)(3782,-2844)
\put(6425,-2258){\makebox(0,0)[lb]{\smash{{\SetFigFont{7}{8.4}{\rmdefault}{\mddefault}{\updefault}{\color[rgb]{0,0,0}$\exp$}%
}}}}
\put(4761,-503){\makebox(0,0)[lb]{\smash{{\SetFigFont{7}{8.4}{\rmdefault}{\mddefault}{\updefault}{\color[rgb]{0,0,0}$\infty$}%
}}}}
\put(4775,-2789){\makebox(0,0)[lb]{\smash{{\SetFigFont{7}{8.4}{\rmdefault}{\mddefault}{\updefault}{\color[rgb]{0,0,0}$0$}%
}}}}
\put(8193,-552){\makebox(0,0)[lb]{\smash{{\SetFigFont{7}{8.4}{\rmdefault}{\mddefault}{\updefault}{\color[rgb]{0,0,0}$\infty$}%
}}}}
\put(8240,-2798){\makebox(0,0)[lb]{\smash{{\SetFigFont{7}{8.4}{\rmdefault}{\mddefault}{\updefault}{\color[rgb]{0,0,0}$0$}%
}}}}
\put(6363,-946){\makebox(0,0)[lb]{\smash{{\SetFigFont{7}{8.4}{\rmdefault}{\mddefault}{\updefault}{\color[rgb]{0,0,0}$\log_k$}%
}}}}
\put(9164,-954){\makebox(0,0)[lb]{\smash{{\SetFigFont{7}{8.4}{\rmdefault}{\mddefault}{\updefault}{\color[rgb]{0,0,0}$J_k$}%
}}}}
\end{picture}%